\documentclass[a4paper,12pt,reqno]{amsart}
\usepackage{amsfonts}
\usepackage{amsmath}
\usepackage{amssymb}
\usepackage[a4paper]{geometry}
\usepackage{mathrsfs}
\usepackage{csquotes}
\usepackage{xcolor}

\usepackage[colorlinks]{hyperref}
\renewcommand\eqref[1]{(\ref{#1})} %Need with hyperref
\numberwithin{equation}{section}
\theoremstyle{plain}
\newtheorem{thm}{Theorem}[section]
\newtheorem{prop}[thm]{Proposition}
\newtheorem{cor}[thm]{Corollary}
\newtheorem{lem}[thm]{Lemma}
 
\theoremstyle{definition}
\newtheorem{defn}[thm]{Definition}
\newtheorem{rem}[thm]{Remark}

\renewcommand{\wp}{\mathfrak S}
\newcommand{\Rn}{\mathbb R^{n}}
\newcommand{\G}{\mathbb G}
\newcommand{\N}{\mathbb N}

\def\R{\mathcal R}

\def\R{\mathcal R}

\def\N{\mathcal N}
\def\e[#1]{{\textrm{e}}^{#1}}

\def\G{{\mathbb G}}

\def\I{\mathfrak{I}}
\def\L{\mathfrak{L}}

\begin{document}

   \title[Critical cases of Sobolev inequalities]
 %  {Critical cases of Sobolev inequalities on graded groups and applications to fractional order nonlinear subelliptic equations}
 {Critical Gagliardo-Nirenberg, Trudinger, Brezis-Gallouet-Wainger inequalities on graded groups and ground states}

 \author[M. Ruzhansky]{Michael Ruzhansky}
\address{
  Michael Ruzhansky:
 \endgraf
    Department of Mathematics: Analysis, Logic and Discrete Mathematics
  \endgraf
    Ghent University, Belgium
   \endgraf
  and
  \endgraf
  School of Mathematical Sciences
    \endgraf
    Queen Mary University of London, United Kingdom
   \endgraf
  {\it E-mail address} {\rm
michael.ruzhansky@ugent.be}
  }

 \author[N. Yessirkegenov]{Nurgissa Yessirkegenov}  
\address{
  Nurgissa Yessirkegenov:
  \endgraf
  Department of Mathematics: Analysis, Logic and Discrete Mathematics
  \endgraf
    Ghent University, Belgium
  \endgraf
   and
  \endgraf
  Institute of Mathematics and Mathematical Modeling, Kazakhstan
    \endgraf
  {\it E-mail address} {\rm nurgissa.yessirkegenov@ugent.be}
  }

\thanks{This research is funded by the Science Committee of the Ministry of Education and Science of the Republic of Kazakhstan (Grant No. AP09058474) and by the FWO Odysseus 1 grant G.0H94.18N: Analysis and Partial Differential Equations. Michael Ruzhansky was supported by the EPSRC grant EP/R003025/2 and by the Methusalem programme of the Ghent University Special
Research Fund (BOF) (Grant number 01M01021).
}


     \begin{abstract} In this paper we investigate critical Gagliardo-Nirenberg, Trudinger-type and Brezis-Gallouet-Wainger inequalities associated with the positive Rockland operators on graded Lie groups, which includes the cases of $\mathbb R^n$, Heisenberg, and general stratified Lie groups. As an application, using the critical Gagliardo-Nirenberg inequality, the existence of least energy solutions of nonlinear Schr\"{o}dinger type equations is obtained. We also express the best constant in the critical Gagliardo-Nirenberg and Trudinger inequalities in the variational form as well as in terms of the ground state solutions of the corresponding nonlinear subelliptic equations. The obtained results are already new in the setting of general stratified Lie groups (homogeneous Carnot groups). Among new technical methods, we also extend Folland's analysis of H\"older spaces from stratified Lie groups to general homogeneous Lie groups.
     \end{abstract}
     \maketitle

       \tableofcontents

       \section{Introduction}
\label{SEC:intro}
Consider the following Trudinger-Moser inequality
\begin{equation}\label{intro:trud}
\int_{\Rn}(\exp(\alpha|f(x)|^{\frac{p}{p-1}})-1)dx\leq C,\;\;1<p<\infty,
\end{equation}
for $f\in L^{p}_{n/p}(\Rn)=(1-\triangle)^{-n/2p}L^{p}(\Rn)$ with $\|f\|_{L^{p}_{n/p}}\leq1$ and for some positive constants $C$ and $\alpha$. This inequality has been generalised in many directions. In bounded domains of $\Rn$ with $p=n\geq2$, we refer to \cite{M79}, \cite{A88}, \cite{CC86}, \cite{F92}, \cite{MP89}, \cite{T67} for the finding of the best exponents in \eqref{intro:trud}, and to \cite{AS07} for the singular version of this inequality. In unbounded domains, we refer to \cite{A88}, \cite{A75}, \cite{OO91}, \cite{Oz95}, \cite{S72}, \cite{Og90}, \cite{AT99}, \cite{Oz95} for Sobolev spaces of fractional order and higher order.

{In \cite{LL13}, the authors developed a rearrangement-free argument without using symmetrization to establish
the Trudinger-Moser inequalities in the unbounded space $\Rn$ including Adams type inequalities
on the higher order derivatives (and fractional order derivatives). Rearrangement fails on the
Heisenberg group or higher order Sobolev spaces. In \cite{LL13} the authors avoid such a rearrangement which
is only available in the first order in the Euclidean spaces. We also refer to \cite{Y14} for a rearrangement free argument on the Heisenberg group, where the author obtained the Trudinger–Moser inequalities when $p=Q$ by gluing local estimates with the cut-off functions. On the Heisenberg group, we also refer to \cite{CL01}, \cite{LLT12} for an analogue of inequality \eqref{intro:trud} on domains of finite measure, and to \cite{LL12}, \cite{LLT14}, \cite{CNGY12} and \cite{Y12} on the entire Heisenberg group as well as to \cite{LL13}, \cite{RY19_Japan} and \cite{RY19_survey} on stratified (Lie) groups. We also refer to \cite{LLZ18} for the results of concentration-compactness type on the Heisenberg group and beyond using the level set argument.}

In this paper, we are interested in obtaining such inequalities on graded Lie groups. We use the strategy developed in \cite{Oz95} and \cite{Oz97} on $\Rn$.

A connected simply connected Lie group $\G$ is called a graded (Lie) group if its Lie algebra admits a gradation. The graded groups form the subclass of homogeneous nilpotent Lie groups admitting homogeneous hypoelliptic left-invariant differential operators (\cite{Miller:80}, \cite{tER:97}, see also a discussion in \cite[Section 4.1]{FR16}). These operators are called Rockland operators from the Rockland conjecture, solved by Helffer and Nourrigat \cite{HN-79}. So, we understand by a Rockland operator {\em any left-invariant homogeneous hypoelliptic differential operator on} $\G$.

In this paper, we are interested in obtaining the inequality \eqref{intro:trud} associated with positive Rockland operators on graded groups. We are also interested to obtain critical Gagliardo-Nirenberg and Brezis-Gallouet-Wainger inequalities. Consequently, we give applications of these inequalities to the nonlinear subelliptic equations. As such, this is essentially the most general framework for such inequalities in the setting of nilpotent Lie groups. Indeed, if a nilpotent Lie group has a left-invariant hypoelliptic differential operator, then the group is graded, see Section \ref{SEC:prelim} for definitions and some details.

From now on we let $\mathcal{R}$ be a positive Rockland operator, that is, a positive left-invariant homogeneous
hypoelliptic invariant differential operator on $\mathbb{G}$ of homogeneous degree $\nu$.
Its powers $\R^{a}$ for any $a>0$ are understood through the functional calculus on the whole of $\G$, extensively analysed in \cite{FR16,FR:Sobolev}. We denote the Sobolev space by $L^{p}_{a}(\G)=L^{p}_{a, \R}(\G)$, for $a>0$, defined by the norm
\begin{equation}\label{norm}
\|u\|_{L^{p}_{a,\R}(\G)}:=\left(\int_{\G}(|\R^{a/\nu}u(x)|^{p}+|u(x)|^{p})dx\right)^{1/p}.
\end{equation}
We refer to \cite[Theorem 4.4.20]{FR16} for the independence of the spaces $L^{p}_{a}(\G)$ of a particular choice of the Rockland operator $\R$.

Thus, in this paper we will show that for a graded group $\mathbb{G}$ of homogeneous dimension $Q$ and for a positive Rockland operator $\mathcal{R}$ of homogeneous degree $\nu$ we have the following results:

\begin{itemize}
\item {\bf (Critical Gagliardo-Nirenberg inequality)} Let $1<p<\infty$. Then we have
\begin{equation}\label{intro:crit_GN_ineq}
\|f\|_{L^{q}(\G)}\leq C_{1}q^{1-1/p}\|\R^{\frac{Q}{\nu p}}f\|_{L^{p}(\G)}^{1-p/q}\|f\|_{L^{p}(\G)}^{p/q}
\end{equation}
for any $q$ with $p\leq q<\infty$ and for any function $f$ from the Sobolev space $L_{Q/p}^{p}(\G)$ on graded group $\G$, where the constant $C_{1}$ depends only on $p$ and $Q$.

\item {\bf (Trudinger inequality with remainders)} Let $1<p<\infty$. Then there exist positive $\alpha$ and $C_{2}$ such that
\begin{equation}\label{intro:Trud_ineq}
\int_{\G}(\exp(\alpha|f(x)|^{p'})-\sum_{0\leq k<p-1, \;k\in\mathbb{N}}\frac{1}{k!}(\alpha|f(x)|^{p'})^{k})dx\leq C_{2}\|f\|^{p}_{L^{p}(\G)}
\end{equation}
holds for any function $f\in L_{Q/p}^{p}(\G)$ with $\|\R^{\frac{Q}{\nu p}}f\|_{L^{p}(\G)}\leq1$, where $1/p+1/p'=1$. Furthermore, we show that \eqref{intro:crit_GN_ineq} and \eqref{intro:Trud_ineq} are actually equivalent and give the relation between their best constants. In \cite{RY18_hypo} using this result we obtained weighted versions of \eqref{intro:Trud_ineq} on graded groups. {In the case $p=Q$, for the best constant $\alpha$ in the weighted Trudinger-Moser inequalities we refer to \cite[Theorem 1.6]{LL12} on the Heisenberg group and to \cite[Theorem G]{LL13} on general stratified groups when $\left[\int_{\G}\left|\nabla_{H} u(\xi)\right|^{Q} d \xi+\tau \int_{\G}|u(\xi)|^{Q} d \xi\right]^{1 / Q}\leq 1$ for any fixed positive real number $\tau$, and to \cite[Theorem 1.1]{LLT14} on the Heisenberg group when $\left\|\nabla_{H} u\right\|_{L^{Q}(\mathbb{H}_{n})} \leq 1$.}

\item {\bf (Brezis-Gallouet-Wainger inequality)}   Let $a,p,q\in \mathbb{R}$ with $1<p,q<\infty$ and $a>Q/q$. Then we have
\begin{equation}\label{intro:BGW1}
\|f\|_{L^{\infty}}\leq C_{3}(1+\log(1+\|\R^{a/\nu} f\|_{L^{q}(\G)}))^{1/p'}
\end{equation}
for any function $f\in L^{p}_{Q/p}(\G)\cap L^{q}_{a}(\G)$ with $\|f\|_{L^{p}_{Q/p}(\G)}\leq1$.

\item {\bf (Existence of ground state solutions)} Let $1<p<q<\infty$. Then the Schr\"{o}dinger type equation \eqref{nonlinear} has a least energy solution $\phi\in L^{p}_{Q/p}(\mathbb{G})$.

Furthermore, we have $d=\L(\phi)$, for the variational problem \eqref{L}-\eqref{d}.

The nonlinear equation \eqref{nonlinear} mentioned above appears naturally in the analysis of the best constants for the above inequalities:

\smallskip
\item {\bf (Best constants in critical Gagliardo-Nirenberg inequality)} Let $1<p<q<\infty$. Let $\phi$ be a least energy solution of \eqref{nonlinear} and let $C_{GN, \R}$ be the smallest positive constant of $C_{1}$ in \eqref{intro:crit_GN_ineq}. Then we have
$$C_{GN, \R}=q^{-q+q/p}\frac{q}{p}\left(\frac{q-p}{p}\right)^{\frac{p-q}{p}}
\|\phi\|_{L^{p}(\mathbb{G})}^{p-q}$$
\begin{equation}\label{intro:sharp1} =
q^{-q+q/p}\frac{q}{p}\left(\frac{q-p}{p}\right)^{\frac{p-q}{p}}
\left(\frac{p^{2}}{q-p}d\right)^{\frac{p-q}{p}}.
\end{equation}
Since \eqref{intro:crit_GN_ineq} and \eqref{intro:Trud_ineq} are equivalent with a relation between constants, this also gives the best constant in Trudinger inequalities.
\end{itemize}

{\em We note that the above results are already new if $\mathbb{G}$ is a stratified group and $\R$ is the (positive) sub-Laplacian on $\mathbb{G}$ (so that also $\nu=2$).}

\smallskip
The paper is organised as follows. In Section \ref{SEC:prelim} we briefly recall main concepts of graded groups and fix the notation. The critical Gagliardo-Nirenberg inequality and Trudinger-type inequality \eqref{intro:trud} are obtained on graded groups in Section \ref{SEC:Trudinger}, where the constant $C$ is given more explicitly. In Section \ref{SEC:Brezis}, we prove the Brezis-Gallouet-Wainger inequalities on graded groups. Finally, applications are given to the nonlinear Schr\"{o}dinger type equations in Section \ref{SEC:applications}.

\section{Preliminaries}
\label{SEC:prelim}

Following Folland and Stein \cite[Chapter 1]{FS-book} and the recent exposition in \cite[Chapter 3]{FR16} we recall that $\mathbb{G}$ is a graded (Lie) group if its Lie algebra $\mathfrak{g}$ admits a gradation
$$\mathfrak{g}=\bigoplus_{\ell=1}^{\infty}\mathfrak{g}_{\ell},$$
where the $\mathfrak{g}_{\ell}$, $\ell=1,2,...,$ are vector subspaces of $\mathfrak{g}$, all but finitely many equal to $\{0\}$, and satisfying
$$[\mathfrak{g}_{\ell},\mathfrak{g}_{\ell'}]\subset \mathfrak{g}_{\ell+\ell'} \;\;\forall \ell, \ell'\in \mathbb{N}.$$

It is called stratified if $\mathfrak{g}_{1}$ generates the whole of $\mathfrak{g}$ through these commutators.

We fix a basis $\{X_{1},\ldots,X_{n}\}$ of a Lie algebra $\mathfrak{g}$ adapted to the gradation. By the exponential mapping $\exp_{\mathbb{G}}:\mathfrak{g}\rightarrow\mathbb{G}$ we get points in $\mathbb{G}$:
$$x=\exp_{\mathbb{G}}(x_{1}X_{1}+\ldots+x_{n}X_{n}).$$
A family of linear mappings of the form
$$D_{r}={\rm Exp}(A \,{\rm ln}r)=\sum_{k=0}^{\infty}
\frac{1}{k!}({\rm ln}(r) A)^{k}$$
is a family of dilations of $\mathfrak{g}$. Here $A$ is a diagonalisable linear operator on $\mathfrak{g}$ with positive eigenvalues. Every $D_{r}$ is a morphism of the Lie algebra $\mathfrak{g}$, i.e., $D_{r}$ is a linear mapping from $\mathfrak{g}$ to itself with the property
$$\forall X,Y\in \mathfrak{g},\, r>0,\;
[D_{r}X, D_{r}Y]=D_{r}[X,Y],$$
as usual $[X,Y]:=XY-YX$ is the Lie bracket. One can extend these dilations through the exponential mapping to the group $\G$ by
\begin{equation}\label{dil_weight}
D_{r}(x)=rx:=(r^{\nu_{1}}x_{1},\ldots,r^{\nu_{n}}x_{n}), \;\;x=(x_{1},\ldots,x_{n})\in\mathbb{G},\;\;r>0,
\end{equation}
where $\nu_{1},\ldots,\nu_{n}$ are weights of the dilations. The sum of these weights
$$
Q:={\rm Tr}\, A=\nu_1+\cdots+\nu_n
$$
is called the homogeneous dimension of $\G$. We also recall that the standard Lebesgue measure $dx$ on $\mathbb{R}^{n}$ is the Haar measure for $\mathbb{G}$ (see, e.g. \cite[Proposition 1.6.6]{FR16}). A {homogeneous quasi-norm} on $\mathbb G$ is
a continuous non-negative function
$$\mathbb{G}\ni x\mapsto |x|\in [0,\infty)$$
with the properties:
\begin{itemize}
\item   $|x^{-1}| = |x|$ for any $x\in \mathbb{G}$,
\item  $|\lambda x|=\lambda |x|$ for any
$x\in \mathbb{G}$ and $\lambda >0$,
\item  $|x|= 0$ if and only if $x=0$.
\end{itemize}

We will use the following polar decomposition for our analysis: there is a (unique) positive Borel measure $\sigma$ on the unit sphere
\begin{equation}\label{EQ:sphere}
\wp:=\{x\in \mathbb{G}:\,|x|=1\},
\end{equation}
such that for any function $f\in L^{1}(\mathbb{G})$ we have
\begin{equation}\label{EQ:polar}
\int_{\mathbb{G}}f(x)dx=\int_{0}^{\infty}
\int_{\wp}f(ry)r^{Q-1}d\sigma(y)dr.
\end{equation}

Let $\widehat{\mathbb{G}}$ be a unitary dual of $\mathbb{G}$ and let $\mathcal{H}_{\pi}^{\infty}$ be the space of smooth vectors for a representation $\pi\in\widehat{\mathbb{G}}$. A Rockland operator $\mathcal{R}$ on $\mathbb{G}$ is a left-invariant differential operator which is homogeneous of positive degree and satisfies the condition:

({\bf Rockland condition}) for each representation $\pi\in\widehat{\mathbb{G}}$, except for the trivial representation, the operator $\pi(\R)$ is injective on $\mathcal{H}_{\pi}^{\infty}$, i.e.,
$$\forall \upsilon \in \mathcal{H}_{\pi}^{\infty}, \;\;\pi(\R)\upsilon=0\Rightarrow \upsilon=0.$$
Here $\pi(\R):=d\pi(\R)$ is the infinitesimal representation of the Rockland operator $\R$ as of an element of the universal enveloping algebra of $\G$.

Different characterisations of such operators have been obtained by Rockland \cite{Rockland} and Beals \cite{Beals-Rockland}. For an extensive presentation about Rockland operators and for the theory of Sobolev spaces on graded groups we refer to \cite{FR:Sobolev} and \cite[Chapter 4]{FR16}, and for the Besov spaces on graded groups we refer to \cite{CR17}.

Since we will not be using the representation theoretic interpretation of these operators in this paper, we define {\em Rockland operators as left-invariant homogeneous hypoelliptic differential operators on $\G$.} This is equivalent to the Rockland condition as it was shown by Helffer and Nourrigat in \cite{HN-79}.

The homogeneous and inhomogeneous Sobolev spaces $\dot{L}^{p}_{a}(\mathbb{G})$ and ${L}^{p}_{a}(\mathbb{G})$ based on the positive left-invariant hypoelliptic differential operator $\R$ have been extensively analysed in \cite{FR:Sobolev} and \cite[Section 4.4]{FR16} to which we refer for the details of their properties. They generalise the Sobolev spaces based on the sub-Laplacian on stratified groups analysed by Folland in \cite{F75}. We refer to the above papers for (non-critical) Sobolev inequalities in the setting of graded groups, and to \cite{RTY17} to the determination of the best constants in non-critical Sobolev and Gagliardo-Nirenberg inequalities on graded groups.

\section{Critical Gagliardo-Nirenberg and Trudinger inequalities}
\label{SEC:Trudinger}
We recall the following Gagliardo-Nirenberg inequality (see \cite[Theorem 3.2]{RTY17}): Let $a\geq0$, $1<p<\frac{Q}{a}$ and $p\leq q\leq\frac{pQ}{Q-ap}$. Then we have for all functions $u$ from the homogeneous Sobolev space $\dot{L}^{p}_{a}(\mathbb{G})$ on the graded group $\G$:
\begin{equation}\label{GN1}
\int_{\mathbb{G}}|u(x)|^{q}dx\leq C \left(\int_{\mathbb{G}}|\mathcal{R}^{\frac{a}{\nu}}u(x)|^{p}dx\right)^{\frac{Q(q-p)}{ap^{2}}}
\left(\int_{\mathbb{G}}|u(x)|^{p}dx\right)^{\frac{apq-Q(q-p)}{ap^{2}}}.
\end{equation}

In this section, we show this inequality for $a=Q/p$, which can be viewed as a critical Gagliardo-Nirenberg inequality. Then, we prove Trudinger-type inequality on graded groups, and we show the equivalence of these two inequalities. We note that another version of the Gagliardo-Nirenberg inequalities was also given in \cite{BFKG-graded}.

In order to prove the critical Gagliardo-Nirenberg inequality, we need to recall the following results.

\begin{thm}[{\cite[Theorem 4.4.28, Part 6]{FR16}}]
\label{Sob-emb-thm}
Let $\mathbb{G}$ be a graded Lie group of homogeneous dimension $Q$. Let $1<p<\infty$ and $a,b,c\in\mathbb{R}$ with $a<c<b$. Then we have
\begin{equation}\label{Th_Sob_emb1} \|f\|_{\dot{L}^{p}_{c}(\G)}\leq C\|f\|_{\dot{L}^{p}_{a}(\G)}^{1-\theta}\|f\|_{\dot{L}^{p}_{b}(\G)}^{\theta}
\end{equation}
for any function $f\in \dot{L}^{p}_{b}(\G)$, where $\theta:=(c-a)/(b-a)$ and the constant $C$ depends only on $a$, $b$ and $c$.
\end{thm}

We will also need the following statement where we refer to Definition \ref{DEF:hom} for the precise definition of the operators of type $\nu$.

\begin{cor}[{\cite[Corollary 3.2.32]{FR16}}]
\label{int_op_cor}
Let $\G$ be a homogeneous Lie group and let $\nu$ be a complex number such that $0\leq{\rm Re}\nu<Q$. Then, all operators of type $\nu$ on $\G$ are $(-\nu)$-homogeneous and extend to a bounded operator from $L^{p}(\G)$ to $L^{q}(\G)$ provided that $1/p-1/q={\rm Re}\nu/Q$ for $1<p\leq q<\infty$.
\end{cor}
Now let us state the critical Gagliardo-Nirenberg inequality.
\begin{thm}\label{crit_GN_thm}
Let $\mathbb{G}$ be a graded Lie group of homogeneous dimension $Q$ and let $\mathcal{R}$ be a positive Rockland operator of homogeneous degree $\nu$. Let $1<p<\infty$. Then we have
\begin{equation}\label{crit_GN_ineq}
\|f\|_{L^{q}(\G)}\leq C_{1}q^{1-1/p}\|\R^{\frac{Q}{\nu p}}f\|_{L^{p}(\G)}^{1-p/q}\|f\|_{L^{p}(\G)}^{p/q}
\end{equation}
for every $q$ with $p\leq q<\infty$ and for every function $f\in L_{Q/p}^{p}(\G)$, where the constant $C_{1}$ depends only on $p$ and $Q$.
\end{thm}
\begin{rem} We note that when $\G=(\Rn,+)$ and $\R=-\triangle$ is the Laplacian, the inequality \eqref{crit_GN_ineq} was obtained in \cite{OO91} for $p=2$, in \cite{KOS92} for $p=n$ and in \cite{Oz95} for general $p$ as in Theorem \ref{crit_GN_thm}.
\end{rem}
\begin{proof}[Proof of Theorem \ref{crit_GN_thm}] We can assume that $f\not\equiv0$. Let us consider the Riesz potential $I_{\lambda}$ with $0<\lambda<Q$ (see e.g. \cite[Section 4.3.4]{FR16} on graded groups and \cite{RSY_Morrey} on general homogeneous groups), given by
\begin{equation}\label{Riesz} (I_{\lambda}f)(x):=\int_{\G}|xy^{-1}|^{\lambda-Q}f(y)dy=(K_{\lambda}\ast f)(x),
\end{equation}
where $K_{\lambda}(x)=|x|^{\lambda-Q}$. Now we decompose, for some $s>0$ to be chosen later,
$$(I_{\lambda}f)(x)=(I^{(1)}_{\lambda}(s)f)(x)+(I^{(2)}_{\lambda}(s)f)(x):=(K^{(1)}_{\lambda,s}\ast f)(x)+(K^{(2)}_{\lambda,s}\ast f)(x),$$
where $K^{(1)}_{\lambda,s}$ and $K^{(2)}_{\lambda,s}$ are defined by
$$
K_{\lambda}(x)=:
\begin{cases} K^{(1)}_{\lambda,s}, \;\;|x|<s;\\
K^{(2)}_{\lambda,s},   \;\;|x|\geq s.\end{cases}$$
Let $1/\widetilde{q}=1/\widetilde{p}-\lambda/Q$ and $1\leq \widetilde{p}<\widetilde{q}<\infty$. Using Young's inequality (see e.g. \cite[Proposition 1.5.2]{FR16}), and introducing polar coordinates $(r,y)=(|x|, \frac{x}{\mid x\mid})\in (0,\infty)\times\wp$ on $\mathbb{G}$, where $\wp$ is the sphere as in \eqref{EQ:sphere}, and by \eqref{EQ:polar}, we have
$$\|I^{(1)}_{\lambda}(s)f\|_{L^{\widetilde{p}}(\G)}\leq
\|K^{(1)}_{\lambda,s}\|_{L^{1}(\G)}\|f\|_{L^{\widetilde{p}}(\G)}=\frac{|\wp|}{\lambda}s^{\lambda}\|f\|_{L^{\widetilde{p}}(\G)}.$$
Similarly, we get
$$\|I^{(2)}_{\lambda}(s)f\|_{L^{\infty}(\G)}\leq
\|K^{(2)}_{\lambda,s}\|_{L^{\widetilde{p}'}(\G)}\|f\|_{L^{\widetilde{p}}(\G)}
=\left(\frac{|\wp|\widetilde{q}}{Q\widetilde{p}'}\right)^{1/\widetilde{p}'}s^{\lambda-Q/\widetilde{p}}\|f\|_{L^{\widetilde{p}}(\G)},$$
where $1/\widetilde{p}+1/\widetilde{p}'=1$ and $|\wp|$ is the $Q-1$ dimensional surface measure of the unit sphere $\wp$.

Then, as in \cite[Section 2, Formula (2.1)]{Oz95}, one can observe that
\begin{multline}\label{Mar0}
\sup_{z>0}z|\{x\in\G:|(I_{\lambda}f)(x)|>z\}|^{1/\widetilde{q}}\\ \leq
2\left(\frac{|\wp|\widetilde{q}}{Q}\right)^{1-1/\widetilde{p}+1/\widetilde{q}}\frac{(\widetilde{p}-1)^{\frac{(\widetilde{p}-1)(\widetilde{q}-\widetilde{p})}{\widetilde{p}\widetilde{q}}}}
{\widetilde{p}^{1-1/\widetilde{p}-(2\widetilde{p}-1)/\widetilde{q}}(\widetilde{q}-\widetilde{p})^{\widetilde{p}/\widetilde{q}}}\|f\|_{L^{\widetilde{p}}(\G)},
\end{multline}
where $1/\widetilde{q}=1/\widetilde{p}-\lambda/Q$, $1\leq \widetilde{p}<\widetilde{q}<\infty$, $0<\lambda<Q$.

If $(1/p_{1},1/q_{1})=(1,1-\lambda/Q)$ with $0<\lambda<Q$, then \eqref{Mar0} implies that
\begin{equation}\label{Mar1}
\sup_{z>0}z|\{x\in\G:|(I_{\lambda}f)(x)|>z\}|^{1/q_{1}}\leq2\left(\frac{|\wp|q_{1}}{Q(q_{1}-1)}\right)^{1/q_{1}}\|f\|_{L^{1}(\G)}.
\end{equation}
If $\left(\frac{1}{p_{2}},\frac{1}{q_{2}}\right)=\left(\frac{1}{p}-\frac{(Q-\lambda p)^{2}}{pQ(pQ+Q-\lambda p)},\frac{Q-\lambda p}{Qp+Q-\lambda p}\right)$ with $0<\lambda<Q$ and $1<p<Q/\lambda$, then \eqref{Mar0} gives that
$$\sup_{z>0}z|\{x\in\G:|(I_{\lambda}f)(x)|>z\}|^{1/q_{2}}$$
\begin{equation}\label{Mar2}
\leq
2\left(\frac{|\wp|q_{2}}{Q}\right)^{1-1/p_{2}+1/q_{2}}
\frac{(p_{2}-1)^{\frac{(p_{2}-1)(q_{2}-p_{2})}{p_{2}q_{2}}}}
{p_{2}^{1-1/p_{2}-(2p_{2}-1)/q_{2}}(q_{2}-p_{2})^{p_{2}/q_{2}}}\|f\|_{L^{p_{2}}(\G)}.
\end{equation}
We fix $p\in(1,\infty)$ and define $\lambda=Q(1/p-1/q)$ for all $q$ with $p<q<\infty$. Then $0<\lambda<Q$, $1<p<Q/\lambda$, $(1/p_{1},1/q_{1})=(1,1-1/p+1/q)$ and $(1/p_{2},1/q_{2})=(1/p-1/q+1/(q+1),1/(q+1))$, so that \eqref{Mar1} and \eqref{Mar2} show that $I_{Q(1/p-1/q)}$ is of weak types $(p_{1},q_{1})$ and $(p_{2},q_{2})$, respectively. Setting $\theta=\theta(q)=(1-1/p)/(1-1/p+1/q-1/(q+1))$, we get $0<\theta<1$, $1/p=(1-\theta)/p_{1}+\theta/p_{2}$ and $1/q=(1-\theta)/q_{1}+\theta/q_{2}$.

As in \cite[Section 2]{Oz95}, we can now use the Marcinkiewicz interpolation theorem to obtain
\begin{equation}\label{I_est1}
\|I_{\lambda}f\|_{L^{q}(\G)}\leq4\left(q+\frac{pq}{q-p}\right)^{1/q}
(M_{1}(q))^{1-\theta}(M_{2}(q))^{\theta}\|f\|_{L^{p}(\G)},
\end{equation}
where $\lambda=Q(1/p-1/q)$, and
$$M_{1}(q)=\left(\frac{|\wp|q_{1}}{Q(q_{1}-1)}\right)^{1/q_{1}},$$
$$M_{2}(q)=\left(\frac{|\wp|q_{2}}{Q}\right)^{1-1/p_{2}+1/q_{2}}
\frac{(p_{2}-1)^{\frac{(p_{2}-1)(q_{2}-p_{2})}{p_{2}q_{2}}}}
{p_{2}^{1-1/p_{2}-(2p_{2}-1)/q_{2}}(q_{2}-p_{2})^{p_{2}/q_{2}}}.$$
Considering the growth property of the right-hand side of \eqref{I_est1} with respect to $q$, one gets
$$\lim_{q\rightarrow\infty}\theta(q)=1,\;\;\lim_{q\rightarrow\infty}M_{1}(q)=\left(\frac{|\wp|p}{Q}\right)^{1-1/p},$$
and
$$\lim_{q\rightarrow\infty}q^{1/p-1}M_{2}(q)=\left(\frac{|\wp|(p-1)}{Qp}\right)^{1-1/p}.$$
So, there is a constant $C$ depending on $p$ and $Q$ for any $q$ with $p<q<\infty$ such that
$$4\left(q+\frac{pq}{q-p}\right)^{1/q}
(M_{1}(q))^{1-\theta}(M_{2}(q))^{\theta}\leq Cq^{1-1/p}$$
holds. Using this and \eqref{I_est1}, we deduce
\begin{equation}\label{I-est2}
\|I_{\lambda}f\|_{L^{q}(\G)}\leq C q^{1-1/p}\|f\|_{L^{p}(\G)},
\end{equation}
for any $q$ with $p<q<\infty$. Since $\R^{-\lambda/\nu}$ is the Riesz potential (see e.g. \cite[Section 4.3.4]{FR16}), then the inequality \eqref{I-est2} with Corollary \ref{int_op_cor} imply that
\begin{equation}\label{I-est3}
\|f\|_{L^{q}(\G)}\leq C q^{1-1/p}\|\R^{(Q/\nu)(1/p-1/q)}f\|_{L^{p}(\G)},
\end{equation}
for $0<\nu<Q$ and for any $q$ with $p<q<\infty$, where $C$ depends only on $p$ and $Q$.

By Theorem \ref{Sob-emb-thm} with $a=Q/p$, $b=0$ and $c=Q(1/p-1/q)$, hence $\theta=\frac{p}{q}$, we have
\begin{equation}\label{I-est4}
\|\R^{(Q/\nu)(1/p-1/q)}f\|_{L^{p}(\G)}\leq C \|\R^{\frac{Q}{\nu p}}f\|_{L^{p}(\G)}^{1-p/q}\|f\|_{L^{p}(\G)}^{p/q},
\end{equation}
which implies \eqref{crit_GN_ineq} in view of \eqref{I-est3}.
\end{proof}
Now we state the Trudinger-type inequality with the remainder estimate on graded groups.
\begin{thm}\label{Trud_thm} Let $\mathbb{G}$ be a graded Lie group of homogeneous dimension $Q$ and let $\mathcal{R}$ be a positive Rockland operator of homogeneous degree $\nu$. Let $1<p<\infty$. Then there exist positive $\alpha$ and $C_{2}$ such that
\begin{equation}\label{Trud_ineq}
\int_{\G}\left(\exp(\alpha|f(x)|^{p'})-\sum_{0\leq k<p-1, \;k\in\mathbb{N}}\frac{1}{k!}(\alpha|f(x)|^{p'})^{k}\right)dx\leq C_{2}\|f\|^{p}_{L^{p}(\G)}
\end{equation}
holds for all functions $f\in L_{Q/p}^{p}(\G)$ with $\|\R^{\frac{Q}{\nu p}}f\|_{L^{p}(\G)}\leq1$, where $1/p+1/p'=1$.
\end{thm}
\begin{rem}\label{rem_Trud_thm} The constant $C_{2}$ can be expressed in terms of the constant $C_{1}=C_{1}(p,Q)$ in \eqref{crit_GN_ineq} as follows
$$C_{2}=C_{2}(\alpha)=\sum_{k\geq p-1,\; k\in\mathbb{N}}\frac{k^{k}}{k!}(p'C_{1}^{p'}\alpha)^{k}.$$
Then, we have \eqref{Trud_ineq} for all $\alpha\in(0, (ep'C_{1}^{p'})^{-1})$ and $C_{2}(\alpha)$.
\end{rem}
\begin{rem} In the case $\G=(\Rn,+)$ and $\R=-\triangle$ is the Laplacian, the inequality \eqref{Trud_ineq} was obtained in \cite{Oz95}. In this abelian case, we can also refer to \cite{Og90} for $p=2$, \cite{OO91} for $p=n=2$, and to \cite{AT99} for $p=n\geq2$.
\end{rem}
\begin{proof}[Proof of Theorem \ref{Trud_thm}] A direct calculation gives
$$\int_{\G}\left(\exp(\alpha|f(x)|^{p'})-\sum_{0\leq k<p-1, \;k\in\mathbb{N}}\frac{1}{k!}(\alpha|f(x)|^{p'})^{k}\right)dx
=\sum_{k\geq p-1, \;k\in\mathbb{N}}\frac{\alpha^{k}}{k!}\int_{\G}|f(x)|^{p'k}dx.$$
Since $k\geq p-1$, we have $p'k\geq p$, then using Theorem \ref{crit_GN_thm} for the above integrals in the last line, we calculate
$$\int_{\G}(\exp(\alpha|f(x)|^{p'})-\sum_{0\leq k<p-1, \;k\in\mathbb{N}}\frac{1}{k!}(\alpha|f(x)|^{p'})^{k})dx$$
$$\leq \sum_{k\geq p-1, \;k\in\mathbb{N}}\frac{\alpha^{k}}{k!} C_{1}^{p'k}(p'k)^{p'k(1-1/p)}\int_{\G}|f(x)|^{p}dx$$
$$=\sum_{k\geq p-1, \;k\in\mathbb{N}}\frac{k^{k}}{k!}(p'C_{1}^{p'}\alpha)^{k}\|f\|^{p}_{L^{p}(\G)},$$
which implies \eqref{Trud_ineq}.
\end{proof}
Now we show that the obtained critical Gagliardo-Nirenberg inequality \eqref{crit_GN_ineq} and Trudinger-type inequality \eqref{Trud_ineq} are actually equivalent on general graded groups. Note that it is already known in $\Rn$, see \cite{Oz97}.
\begin{thm}\label{equiv}
The inequalities \eqref{crit_GN_ineq} and \eqref{Trud_ineq} are equivalent. Furthermore, we have
\begin{equation}\label{equiv_identity}
\frac{1}{\widetilde{\alpha}p'e}=A^{p'}=B^{p'},
\end{equation}
where
\begin{equation}\label{alphaAB}
\begin{split}
&\widetilde{\alpha}=\sup\{\alpha>0; \exists C_{2}=C_{2}(\alpha):\eqref{Trud_ineq} \textrm{ holds }\forall f\in L_{Q/p}^{p}(\G)
\textrm{ with } \|\R^{\frac{Q}{\nu p}}f\|_{L^{p}(\G)}\leq1\},\\
&A=\inf\left\{C_{1}>0; \exists r=r(C_{1}) \textrm{ with } r\geq p:\eqref{crit_GN_ineq}\textrm{ holds }\forall f\in L_{Q/p}^{p}(\G), \forall q \textrm{ with } r\leq q<\infty\right\},\\
&B=\limsup_{q\rightarrow \infty}\frac{\|f\|_{L^{q}(\G)}}{q^{1-1/p}\|\R^{\frac{Q}{\nu p}}f\|_{L^{p}(\G)}^{1-p/q}\|f\|_{L^{p}(\G)}^{p/q}}.
\end{split}
\end{equation}
\end{thm}
\begin{proof}[Proof of Theorem \ref{equiv}] We have already shown that \eqref{crit_GN_ineq} implies \eqref{Trud_ineq} in the proof of Theorem \ref{Trud_thm}. By Remark \ref{rem_Trud_thm}, \eqref{alphaAB} and taking into account that $A\geq B$, we note that $\alpha<(ep'C_{1}^{p'})^{-1}$ implies $\alpha<(ep'B^{p'})^{-1}$, that is, \eqref{crit_GN_ineq} implies \eqref{Trud_ineq} with $\widetilde{\alpha}\geq(ep'B^{p'})^{-1}$.

Now let us show \eqref{Trud_ineq}$\Rightarrow$\eqref{crit_GN_ineq} with $\widetilde{\alpha}\leq(ep'A^{p'})^{-1}$. Since $\|\R^{\frac{Q}{\nu p}}f\|_{L^{p}(\G)}\leq1$, replacing $f$ by $f/\|\R^{\frac{Q}{\nu p}}f\|_{L^{p}(\G)}$ in \eqref{Trud_ineq} we obtain
\begin{multline}\label{Trud_ineq_2}
\int_{\G}\left(\exp\left(\frac{\alpha|f(x)|^{p'}}{\|\R^{\frac{Q}{\nu p}}f\|_{L^{p}(\G)}^{p'}}\right)-\sum_{0\leq k<p-1, \;k\in\mathbb{N}}\frac{1}{k!}\left(\frac{\alpha|f(x)|^{p'}}{\|\R^{\frac{Q}{\nu p}}f\|_{L^{p}(\G)}^{p'}}\right)^{k}\right)dx  \\
\leq C_{2}\frac{\|f\|^{p}_{L^{p}(\G)}}{\|\R^{\frac{Q}{\nu p}}f\|^{p}_{L^{p}(\G)}}.
\end{multline}
Then, we see that for any $\varepsilon$ with $0<\varepsilon<\widetilde{\alpha}$ there is $C_{\varepsilon}$ such that
\begin{multline}\label{Trud_ineq_3}
\int_{\G}\left(\exp\left(\frac{(\widetilde{\alpha}-\varepsilon)|f(x)|^{p'}}{\|\R^{\frac{Q}{\nu p}}f\|_{L^{p}(\G)}^{p'}}\right)-\sum_{0\leq k<p-1, \;k\in\mathbb{N}}\frac{1}{k!}\left(\frac{(\widetilde{\alpha}-\varepsilon)|f(x)|^{p'}}{\|\R^{\frac{Q}{\nu p}}f\|_{L^{p}(\G)}^{p'}}\right)^{k}\right)dx
\\
\leq C_{\varepsilon}\frac{\|f\|^{p}_{L^{p}(\G)}}{\|\R^{\frac{Q}{\nu p}}f\|^{p}_{L^{p}(\G)}}
\end{multline}
holds for all $f\in L_{Q/p}^{p}(\G)$. It follows that
$$
\int_{\G}\sum_{k\geq p-1, \;k\in\mathbb{N}}\frac{1}{k!}\left(\frac{(\widetilde{\alpha}-\varepsilon)|f(x)|^{p'}}{\|\R^{\frac{Q}{\nu p}}f\|_{L^{p}(\G)}^{p'}}\right)^{k}dx\leq C_{\varepsilon}\frac{\|f\|^{p}_{L^{p}(\G)}}{\|\R^{\frac{Q}{\nu p}}f\|^{p}_{L^{p}(\G)}},$$
that is,
\begin{equation}\label{Trud_ineq_3_2}
\|f\|_{L^{p'k}(\G)}\leq (C_{\varepsilon}k!)^{1/p'k}(\widetilde{\alpha}-\varepsilon)^{-1/p'}
\|\R^{\frac{Q}{\nu p}}f\|^{1-(p-1)/k}_{L^{p}(\G)}\|f\|^{(p-1)/k}_{L^{p}(\G)}
\end{equation}
for all $k\in\mathbb{N}$ with $k\geq p-1$. Let $q>p$ and $p'k\leq q <p'(k+1)$. Then, by interpolating this between $L^{p'k}(\G)$ and $L^{p'(k+1)}(\G)$ and taking into account $(k+1)!\leq \Gamma(2+q/p')$, we get
\begin{equation}\label{Trud_ineq_4}
\|f\|_{L^{q}(\G)}\leq (C_{\varepsilon}\Gamma(2+q/p'))^{1/p'k}(\widetilde{\alpha}-\varepsilon)^{-1/p'}
\|\R^{\frac{Q}{\nu p}}f\|^{1-p/q}_{L^{p}(\G)}\|f\|^{p/q}_{L^{p}(\G)},
\end{equation}
where $\Gamma$ is the gamma function. Applying here the Stirling's formula and $p'k\geq q-p'$, we obtain that there is $r$ such that
\begin{equation}\label{crit_GN_ineq2}
\|f\|_{L^{q}(\G)}\leq ((p'e(\widetilde{\alpha}-\varepsilon))^{-1/p'}+\delta)q^{1-1/p}\|\R^{\frac{Q}{\nu p}}f\|_{L^{p}(\G)}^{1-p/q}\|f\|_{L^{p}(\G)}^{p/q}
\end{equation}
holds for all $f\in L_{Q/p}^{p}(\G)$ and $q$ with $r\leq q<\infty$. Thus, $A\leq (p'e(\widetilde{\alpha}-\varepsilon))^{-1/p'}+\delta$, then by arbitrariness of $\varepsilon$ and $\delta$ we obtain $\widetilde{\alpha}\leq(ep'A^{p'})^{-1}$.

This completes the proof of Theorem \ref{equiv}.
\end{proof}
\section{Brezis-Gallouet-Wainger inequalities}
\label{SEC:Brezis}
In this section we investigate Brezis-Gallouet-Wainger inequalities, which concern the limiting case of the Sobolev estimates (see \cite{B82}, \cite{BG80} and \cite{BW80}). As part of the proof we extend the analysis of Folland \cite{F75} related to H\"older spaces from the setting of stratified to general homogeneous groups. For the background analysis on homogeneous groups we refer to Folland and Stein's fundamental book \cite{FS-book} as well as to a more recent treatment in \cite{FR16}.

We recall some definitions from \cite{F75}, see also \cite[Chapter 3]{FR16}.

\begin{defn}\label{DEF:hom}
Let $\G$ be a nilpotent Lie group and let $\lambda$ be a complex number.
\begin{itemize}
\item A measurable function $f$ on $\G$ is called homogeneous of degree $\lambda$ if $f\circ D_{r}=r^{\lambda}f$ for all positive $r>0$, where $D_{r}$ is the family of dilations on $\G$.
\item A distribution $\tau\in\mathcal{D}'$ is called homogeneous of degree $\lambda$ if $\langle\tau,\phi\circ D_{r}\rangle=r^{-Q-\lambda}\langle\tau,\phi\rangle$ for all $\phi\in\mathcal{D}$ and all positive $r>0$.
\item A distribution which is smooth away from the origin and homogeneous of degree $\lambda-Q$ is called a kernel of type $\lambda$ on $\G$.
\end{itemize}
\end{defn}
We also need the following results:
\begin{prop}[{\cite[Proposition 1.4]{F75}}]
\label{Fol_prop1}
Let $\G$ be a nilpotent Lie group and let $|\cdot|$ be a homogeneous quasi-norm on $\G$. Then, there is a positive constant $C$ such that
$$|xy|\leq C(|x|+|y|),\;\;\forall x,y\in \G.$$
\end{prop}
\begin{prop}[{\cite[Proposition 1.15]{F75}}]
\label{Fol_prop2}
Let $\G$ be a nilpotent Lie group and let $|\cdot|$ be a homogeneous quasi-norm on $\G$. For any $f\in C^{2}(\G\backslash\{0\})$ homogeneous of degree $\lambda\in\mathbb{R}$, there are constants $C,\varepsilon>0$ such that
\begin{equation}\label{Fol_prop2_ineq1}
|f(xy)-f(x)|\leq C|y||x|^{\lambda-1}\;\textrm{whenever}\;|y|\leq\frac{1}{2}|x|,
\end{equation}
\begin{equation}\label{Fol_prop2_ineq2}
|f(xy)+f(xy^{-1})-2f(x)|\leq C|y|^{2}|x|^{\lambda-2}\;\textrm{whenever}\;|y|\leq\varepsilon|x|.
\end{equation}
\end{prop}
Let us now state the first main result of this section.
\begin{thm}\label{BGW_thm} Let $\mathbb{G}$ be a graded Lie group of homogeneous dimension $Q$ and let $\mathcal{R}$ be a positive Rockland operator of homogeneous degree $\nu$. Let $a,p,q\in \mathbb{R}$ with $1<p,q<\infty$ and $a>Q/q$. Then there exists $C_{3}>0$ such that we have
\begin{equation}\label{BGW1}
\|f\|_{L^{\infty}(\G)}\leq C_{3}(1+\log(1+\|\R^{a/\nu} f\|_{L^{q}(\G)}))^{1/p'}
\end{equation}
for all functions $f\in L^{p}_{Q/p}(\G)\cap L^{q}_{a}(\G)$ with $\|f\|_{L^{p}_{Q/p}(\G)}\leq1$.
\end{thm}
\begin{rem} In the case $\G=(\Rn,+)$ and $\R=-\triangle$ is the Laplacian, the inequality \eqref{BGW1} was obtained in \cite{BW80} by employing Fourier transform methods, and in \cite{E89} for $n/p, m\in \mathbb{Z}$, and in \cite{Oz95} for the general case without using the Fourier transform.
\end{rem}
We recall the Lipschitz spaces and obtain an estimate on nilpotent group $\G$, which will be used in the proof of Theorem \ref{BGW_thm}. So, let $C_{b}(\G)$ be the space of bounded continuous functions on $\G$. Then we define
$$\Gamma_{\alpha}:=\{f\in C_{b}(\G):|f|_{\alpha}:=\sup_{x,y}|f(xy)-f(x)|/|y|^{\alpha}<\infty\}$$
for $0<\alpha<1$, and when $\alpha=1$, we define
$$\Gamma_{1}:=\{f\in C_{b}(\G):|f|_{1}:=\sup_{x,y}|f(xy)+f(xy^{-1})-2f(x)|/|y|<\infty\}.$$
Note that the Lipschitz space $\Gamma_{\alpha}$ with $0<\alpha\leq1$ is a Banach space with norm $\|f\|_{\Gamma_{\alpha}(\G)}=\|f\|_{L^{\infty}(\G)}+|f|_{\alpha}$.
\begin{lem}\label{Lip_lem}Let $\mathbb{G}$ be a homogeneous Lie group of homogeneous dimension $Q$. Let $0<\lambda<Q$, $0<\alpha\leq1$ and $1<p\leq\infty$ with $\alpha=\lambda-Q/p$. Let $K$ be a kernel of type $\lambda$. Then we have
\begin{equation}\label{lem_Holder}
|Tf|_{\alpha}\leq C \|f\|_{L^{p}(\G)}
\end{equation}
for the mapping $T:f\mapsto f\ast K$.
\end{lem}
\begin{proof}[Proof of Lemma \ref{Lip_lem}] First, let us consider the case $0<\alpha<1$. We know that
\begin{equation}\label{lem_Holder1}
Tf(xy)-Tf(x)=\int_{\G}f(xz^{-1})(K(zy)-K(z))dz.
\end{equation}
As in the case of stratified groups (see \cite[Theorem 5.14]{F75}), we write it as follows
$$Tf(xy)-Tf(x)=\int_{|z|>2|y|}f(xz^{-1})(K(zy)-K(z))dz$$
\begin{equation}\label{lem_Holder2}
+\int_{|z|\leq2|y|}f(xz^{-1})(K(zy)-K(z))dz=:I_{1}+I_{2}.
\end{equation}
Then, we estimate $I_{1}$ using \eqref{Fol_prop2_ineq1} and H\"{o}lder's inequality
\begin{equation}\label{lem_Holder3}
\begin{split}
|I_{1}|&\leq C\|f\|_{L^{p}(\G)}\left(\int_{|z|>2|y|}|y|^{p'}|z|^{(\lambda-Q-1)p'}dz\right)^{1/p'}\\&
\leq C\|f\|_{L^{p}(\G)}|y|(2|y|)^{\lambda-Q-1+Q/p'}\leq C\|f\|_{L^{p}(\G)}|y|^{\alpha}.
\end{split}
\end{equation}
Now for $I_{2}$, by Proposition \ref{Fol_prop1} and $|z|\leq2|y|$ there exists a constant $M_{1}\geq2$ such that $|zy|\leq M_{1}|y|$, so by H\"{o}lder's inequality we have
\begin{equation}\label{lem_Holder4}
\begin{split}
|I_{2}|&\leq \|f\|_{L^{p}(\G)}\left(2\int_{|z|\leq M_{1}|y|}|K(z)|^{p'}dz\right)^{1/p'}\\&
\leq C\|f\|_{L^{p}(\G)}\left(\int_{|z|\leq M_{1}|y|}|z|^{(\lambda-Q)p'}dz\right)^{1/p'}\leq C\|f\|_{L^{p}(\G)}|y|^{\alpha}.
\end{split}
\end{equation}
Combining \eqref{lem_Holder3} and \eqref{lem_Holder4}, we obtain \eqref{lem_Holder} for $0<\alpha<1$.

Now, in the case $\alpha=1$, we prove it in the same way as above but using \eqref{Fol_prop2_ineq2}. So, we write for some positive number $M_{2}$:
$$
Tf(xy)+Tf(xy^{-1})-2Tf(x)=\int_{\G}f(xz^{-1})(K(zy)+K(zy^{-1})-2K(z))dz$$
$$=\int_{|z|>M_{2}|y|}f(xz^{-1})(K(zy)+K(zy^{-1})-2K(z))dz$$
\begin{equation}\label{lem_Holder5}
+\int_{|z|\leq M_{2}|y|}f(xz^{-1})(K(zy)+K(zy^{-1})-2K(z))dz=:I_{3}+I_{4}.
\end{equation}
For $I_{3}$, using \eqref{Fol_prop2_ineq2} and H\"{o}lder's inequality one calculates
\begin{equation}\label{lem_Holder6}
\begin{split}
|I_{3}|&\leq C\|f\|_{L^{p}(\G)}\left(\int_{|z|>M_{2}|y|}|y|^{2p'}|z|^{(\lambda-Q-2)p'}dz\right)^{1/p'}\\&
\leq C\|f\|_{L^{p}(\G)}|y|^{2}(M_{2}|y|)^{\lambda-Q-2+Q/p'}\leq C\|f\|_{L^{p}(\G)}|y|,
\end{split}
\end{equation}
since $\alpha=1$.
Now for $I_{4}$, by Proposition \ref{Fol_prop1} and $|z|\leq M_{2}|y|$ there exists a constant $M_{3}\geq M_{2}$ such that $|zy|\leq M_{3}|y|$ and $|zy^{-1}|\leq M_{3}|y|$, so by H\"{o}lder's inequality we have
\begin{equation}\label{lem_Holder7}
\begin{split}
|I_{4}|&\leq \|f\|_{L^{p}(\G)}\left(4\int_{|z|\leq M_{3}|y|}|K(z)|^{p'}dz\right)^{1/p'}\\&
\leq C\|f\|_{L^{p}(\G)}\left(\int_{|z|\leq M_{3}|y|}|z|^{(\lambda-Q)p'}dz\right)^{1/p'}\leq C\|f\|_{L^{p}(\G)}|y|.
\end{split}
\end{equation}
From \eqref{lem_Holder6} and \eqref{lem_Holder7}, we obtain \eqref{lem_Holder} for $\alpha=1$.
\end{proof}
Now we are ready to prove Theorem \ref{BGW_thm}.
\begin{proof}[Proof of Theorem \ref{BGW_thm}] Let us first prove \eqref{BGW1} when $a-Q/q=:\alpha\in (0,1)$. Since $\R^{-a/\nu}$ is the Riesz potential (see e.g. \cite[Section 4.3.4]{FR16}), by Lemma \ref{Lip_lem} we have \begin{equation}\label{BGW2} |f(xy)-f(x)|\leq C\|\R^{a/\nu}f\|_{L^{q}(\G)}|y|^{\alpha}.
\end{equation}
Let $z\in\G$ with $|z|\leq1$ and $0<\beta<e^{-p}$. Then, \eqref{BGW2} implies that
\begin{equation}\label{BGW3} |f(x)-f(x\beta z)|\leq C\beta^{\alpha}\|\R^{a/\nu}f\|_{L^{q}(\G)}.
\end{equation}
Applying H\"{o}lder's inequality and Theorem \ref{crit_GN_thm}, one gets
\begin{equation}\label{BGW4}
\begin{split}
\int_{|z|\leq1}|f(x\beta z)|dz&\leq\left(\frac{|\wp|}{Q}\right)^{1/r'}\left(\int_{|z|\leq1}|f(x\beta  z)|^{r}dz\right)^{1/r}\leq \left(\frac{|\wp|}{Q}\right)^{1/r'} \beta^{-Q/r} \|f\|_{L^{r}(\G)}\\&
\leq C \beta^{-Q/r} r^{1/p'}\|f\|_{L^{p}_{Q/p}(\G)}\leq C \beta^{-Q/r} r^{1/p'}
\end{split}
\end{equation}
for any $r$ with $p\leq r<\infty$. We take $r=\log(1/\beta)$ to obtain \begin{equation}\label{BGW4_1}
\int_{|z|\leq1}|f(x\beta z)|dz\leq C\left(\log\left(\frac{1}{\beta}\right)\right)^{1/p'}
\end{equation}
for any $\beta$ with $0<\beta<e^{-p}$, where the constant $C$ depends only on $p$ and $Q$.

Then, by a direct calculation, one has
\begin{equation*}
    \begin{split}
     |f(x)|=|f(x)|\cdot\frac{Q}{|\wp|}\int_{|z|\leq1}dz&\leq \frac{Q}{|\wp|}\int_{|z|\leq1}(|f(x)-f(x\beta z)|+|f(x\beta z)|)dz\\&
     \leq C\beta^{\alpha}\|\R^{a/\nu}f\|_{L^{q}(\G)}+C\left(\log\left(\frac{1}{\beta}\right)\right)^{1/p'},
    \end{split}
\end{equation*}
which implies \eqref{BGW1}, after setting
$\beta=1/\left(e^{p}+\|\R^{a/\nu}f\|^{1/\alpha}_{L^{q}(\G)}\right).$

Now it remains to consider the case $a-Q/q\geq1$. Let $a_{0}\in\mathbb{R}$ with $Q/q+1>a_{0}>Q/q$. As in \cite[Section 2]{CR17} (see also \cite[Section 3]{CR16}), we consider the operator $\chi_{L}(\R)$ for every $L>0$, defined by functional calculus, where $\chi_{L}$ is the characteristic function of $[0,L]$. Then, taking into account these and using \eqref{BGW1} for $a\geq Q/q+1>a_{0}>Q/q$ as already proved, we have
\begin{equation}\label{BGW5}
\begin{split}
\|(1-\chi_{L}(\R))f\|_{L^{\infty}(\G)}&\leq C(1+\log(1+\|\R^{a_{0}/\nu}((1-\chi_{L}(\R))f)\|_{L^{q}(\G)}))^{1/p'}\\&
\leq C(1+\log(1+\|\R^{a/\nu}((1-\chi_{L}(\R))f)\|_{L^{q}(\G)}))^{1/p'}.
\end{split}
\end{equation}
For $\chi_{L}(\R)f$, we use the Nikolskii inequality on graded groups \cite[Theorem 2.1]{CR17}:
\begin{equation}\label{BGW6}
\|\chi_{L}(\R)f\|_{L^{\infty}(\G)}\lesssim \|\chi_{L}(\R)f\|_{L^{p}}\lesssim \|f\|_{L^{p}(\G)}\leq1.
\end{equation}
Finally, we obtain
\begin{equation*}
\begin{split}\|f\|_{L^{\infty}(\G)}&\leq \|(1-\chi_{L}(\R))f\|_{L^{\infty}(\G)}+\|\chi_{L}(\R)f\|_{L^{\infty}(\G)}\\&\leq C(1+\log(1+\|\R^{a/\nu}f\|_{L^{q}(\G)}))^{1/p'},
\end{split}
\end{equation*}
completing the proof.
\end{proof}
We can also obtain the following estimate using Theorem \ref{crit_GN_thm}:
\begin{thm} \label{BW_thm} Let $\mathbb{G}$ be a graded Lie group of homogeneous dimension $Q$ and let $1<p<\infty$. Then we have
\begin{equation}\label{BW1}
\int_{\Omega}|f(x)|dx\leq C_{4} \|f\|_{L^{p}_{Q/p}(\G)}|\Omega|(1+|\log|\Omega||)^{1/p'}
\end{equation}
for any function $f\in L^{p}_{Q/p}(\G)$ and for any Lebesgue measurable set $\Omega$ with $|\Omega|<\infty$, where the constant $C_{4}$ depends only on $p$ and $Q$, and $|\Omega|$ denotes the Lebesgue measure of $\Omega$.
\end{thm}
\begin{rem} When $\G=(\Rn,+)$ and $\R=-\triangle$ is the Laplacian, the inequality \eqref{BW1} was obtained in \cite[Lemma 2]{BW80} and in \cite[Theorem 3]{Oz95}. In \cite{BW80}, using this estimate and Morrey's technique the authors proved the Brezis-Wainger inequality: for any function $f\in L^{p}_{n/p+1}(\Rn)$ and for each $x,y\in \Rn$, the inequality
$$|f(x)-f(y)|\leq C\|f\|_{L^{p}_{n/p+1}(\Rn)}|x-y|(1+|\log|x-y||)^{1/p'},\;\;1<p<\infty,$$
holds true, where the constant $C$ depends only on $n$ and $p$.
\end{rem}
\begin{proof}[Proof of Theorem \ref{BW_thm}] Using H\"{o}lder's inequality, we get
\begin{equation}\label{BW2}\int_{\Omega}|f(x)|dx\leq |\Omega|^{1/q'}\|f\|_{L^{q}(\G)}.
\end{equation}
Then, by Theorem \ref{crit_GN_thm}, we can deduce that
\begin{equation}\label{BW3}
\int_{\Omega}|f(x)|dx\leq C|\Omega|^{1/q'}q^{1/p'}\|f\|_{L^{p}_{Q/p}(\G)},
\end{equation}
for $p\leq q <\infty$ with the constant $C$ depending only $p$ and $Q$. Plugging $q=p$ into \eqref{BW2} for $|\Omega|>e^{-p}$, hence $|\Omega|^{1/q'}=|\Omega|^{1-1/p}<e|\Omega|$, one gets
\begin{equation}\label{BW4}\int_{\Omega}|f(x)|dx\leq e |\Omega|\|f\|_{L^{p}(\G)}.
\end{equation}
Now plugging $q=\log(1/|\Omega|)$ into \eqref{BW3} for $0<|\Omega|\leq e^{-p}$, we have
\begin{equation}\label{BW5}
\int_{\Omega}|f(x)|dx\leq C|\Omega||\log|\Omega||^{1/p'}\|f\|_{L^{p}_{Q/p}(\G)}.
\end{equation}
Combining \eqref{BW4} and \eqref{BW5} we obtain \eqref{BW1}.
\end{proof}
\section{Best constants and nonlinear Schr\"{o}dinger type equations}
\label{SEC:applications}
In this section using the critical case of the Gagliardo-Nirenberg inequality \eqref{crit_GN_ineq} we show the existence of least energy solutions for nonlinear Schr\"{o}dinger type equations, and we obtain a sharp expression for the smallest positive constant $C_{1}$ in \eqref{crit_GN_ineq}. For non-critical case on nilpotent Lie group, when $a\neq Q/p$ in the inequality \eqref{GN1}, similar results were obtained in \cite{CR13} on the Heisenberg group and in \cite{RTY17} on graded groups.

Let $\R$ be a positive Rockland operator of homogeneous degree $\nu$, on a graded group $\G$ of homogeneous dimension $Q$. Let $1<p<q<\infty$.
We consider the nonlinear equation with the power nonlinearity
\begin{equation}\label{nonlinear}
\mathcal{R}^{\frac{Q}{\nu p}}(|\mathcal{R}^{\frac{Q}{\nu p}}u|^{p-2}\mathcal{R}^{\frac{Q}{\nu p}}u)+
|u|^{p-2}u=|u|^{q-2}u, \quad u\in L^{p}_{Q/p}(\mathbb{G}).
\end{equation}
Such an equation appears naturally in the analysis of best constants of the established critical inequalities. We prove the existence of least energy solutions to \eqref{nonlinear} and their relation to the best constants in the Gagliardo-Nirenberg inequalities and hence, in view of Theorem \ref{equiv}, also to the best constants in the Trudinger inequalities.

For example, for $p=2$, if $\mathbb{G}$ is a stratified Lie group of homogeneous dimension $Q$ and $\R=\mathcal{L}$ is the positive sub-Laplacian, equation \eqref{nonlinear} becomes
\begin{equation}\label{EQ:sL}
\mathcal{L}^{\frac{Q}{2}}u+
u=|u|^{q-2}u, \quad u\in L^{2}_{Q/2}(\mathbb{G}),
\end{equation}
and if $p\not=2$, equation \eqref{nonlinear} can be regarded as the $p$-sub-Laplacian version of \eqref{EQ:sL}. In the Euclidean case, when $Q=2$, $q=4$ and $\L=-\Delta$, one can note that if $u(x)$ is a solution of \eqref{EQ:sL} then the function $w(x)=u(x)e^{it/2}$ solves the following nonlinear Schr\"{o}dinger equation
$$2iw_{t}+\Delta w +|w|^{2}w=0, \;\;t\in \mathbb{R}^{+},\;\;x\in \mathbb{R}^{2}.$$
Such equations arise in modeling the propagation of a thin electromagnetic beam through a medium with an index of refraction dependent on the field intensity (see for example \cite{CGT64} and \cite[Section V]{Wei83}).

Now, let us give some notations and definitions for this section.

\begin{defn}
A function $u\in L^{p}_{Q/p}(\mathbb{G})$ is said to be a solution of \eqref{nonlinear} if and only if for any function $\psi\in L^{p}_{Q/p}(\mathbb{G})$ the equality
$$\int_{\mathbb{G}}(|\mathcal{R}^{\frac{Q}{\nu p}}u(x)|^{p-2}\mathcal{R}^{\frac{Q}{\nu p}}
u(x)\overline{\mathcal{R}^{\frac{Q}{\nu p}}\psi(x)}dx+
\int_{\mathbb{G}}(|u(x)|^{p-2}u(x)\overline{\psi(x)}$$
\begin{equation}\label{solution}
-|u(x)|^{q-2}u(x)\overline{\psi(x)})dx=0
\end{equation}
holds true.
\end{defn}

We define the functionals $\mathfrak{L}:L^{p}_{Q/p}(\mathbb{G})\to \mathbb R$ and $\mathfrak{I}:L^{p}_{Q/p}(\mathbb{G})\to \mathbb R$ acting on $L^{p}_{Q/p}(\mathbb{G})\cap L^{q}(\G)$ as follows:
\begin{equation}\label{L}
\mathfrak{L}(u):=\frac{1}{p}\int\limits_{\mathbb{G}}|\mathcal{R}^{\frac{Q}{\nu p}}u(x)|^{p}dx+
\frac{1}{p}\int\limits_{\mathbb{G}}|u(x)|^{p}dx
-\frac{1}{q}\int\limits_{\mathbb{G}}|u(x)|^{q}dx
\end{equation}
and
\begin{equation}\label{I}
\mathfrak{I}(u):=\int\limits_{\mathbb{G}}(|\mathcal{R}^{\frac{Q}{\nu p}}u(x)|^{p}+
|u(x)|^{p}-|u(x)|^{q})dx.
\end{equation}

We use the Nehari set
\begin{equation}\label{N}
\mathcal{N}:=\{u\in L^{p}_{Q/p}(\mathbb{G})\ \backslash\{0\}: \I(u)=0\}
\end{equation}
and denote
\begin{equation}\label{d}
d:=\inf\{\L(u):u\in\mathcal{N}\}.
\end{equation}

\begin{defn}\label{def}
Let
$$\Gamma=\{\phi\in L^{p}_{Q/p}(\mathbb{G}):\L'(\phi)=0\;\;{\rm and}\;\;\phi\neq0\}$$
and
$$\mathcal{G}=\{u\in \Gamma: \L(u)\leq \L(\upsilon) \;\;{\rm for \;\;any}\;\;\upsilon\in\Gamma\}$$
be the set of the solutions and the set of least energy solutions of \eqref{nonlinear}, respectively.
\end{defn}

\begin{thm}\label{thm1} Let $\mathbb{G}$ be a graded Lie group of homogeneous dimension $Q$, let $1<p<q<\infty$. Then the Schr\"{o}dinger type equation \eqref{nonlinear} has a least energy solution $\phi\in L^{p}_{Q/p}(\mathbb{G})$.

Furthermore, we have $d=\L(\phi)$.
\end{thm}

In the sequel, we assume $1<p<q<\infty$.

Let us state and prove the following lemmas, which will be used in the proof of Theorem \ref{thm1}.

\begin{lem}\label{LM: 2.1}
For any function $u\in L^{p}_{Q/p}(\mathbb{G})\setminus\{0\}$, there is a unique $\mu_{u}>0$ such that $\mu_{u}u\in\mathcal{N}$. Moreover, we have $0<\mu_{u}<1$ when $\I(u)<0$.
\end{lem}

\begin{proof}[Proof of Lemma \ref{LM: 2.1}]
For any $u\in L^{p}_{Q/p}(\mathbb{G})\setminus\{0\}$ and
\begin{equation}\label{mu}
\mu_{u}=\|u\|_{L^{p}_{Q/p}(\mathbb{G})}^{\frac{p}{q-p}} \|u\|^{-\frac{q}{q-p}}_{L^{q}(\mathbb{G})},
\end{equation}
we note that $\mu_{u}u\in\mathcal{N}$. It is easy to see that this $\mu_{u}$ is unique. Then, by \eqref{mu} one gets $0<\mu_{u}<1$ provided that  $\|u\|_{L^{p}_{Q/p}(\mathbb{G})}^{p}<\|u\|^{q}_{L^{q}(\mathbb{G})}$.
\end{proof}

\begin{lem}\label{LM: 2.2}
For all functions $u\in\mathcal{N}$, we have $\underset{u}{\rm inf}\|u\|_{L^{p}_{Q/p}(\mathbb{G})}>0$.
\end{lem}
\begin{proof}[Proof of Lemma \ref{LM: 2.2}]
Using \eqref{crit_GN_ineq} we calculate
\begin{equation*}
\begin{split}
\|u\|_{L^{p}_{Q/p}(\mathbb{G})}^{p}=\|u\|^{q}_{L^{q}(\mathbb{G})}&\leq C q^{q-q/p} \|\mathcal{R}^{\frac{Q}{\nu p}} u\|^{q-p}_{L^{p}(\mathbb{G})} \|u\|^{p}_{L^{p}(\mathbb{G})} \\
&\leq C q^{q-q/p} \|u\|^{q-p}_{L^{p}_{Q/p}(\mathbb{G})} \|u\|^{p}_{L^{p}_{Q/p}(\mathbb{G})}\\
&\leq C \|u\|_{L^{p}_{Q/p}(\mathbb{G})}^{q}
\end{split}
\end{equation*}
for all $u\in\mathcal{N}$. From this we obtain $\|u\|_{L^{p}_{Q/p}(\mathbb{G})}^{q-p}\geq C^{-1}$, which implies $\|u\|_{L^{p}_{Q/p}(\mathbb{G})}\geq \kappa$ for any function $u\in\mathcal{N}$ after setting $\kappa=C^{-\frac{1}{q-p}}>0$.
\end{proof}
Let us show the following Rellich-Kondrachev type lemma. On the Heisenberg group a similar result was obtained by Garofalo and Lanconelli \cite{GL-1992}.

\begin{lem}\label{LM: 2.3}
Let $1<p< q<\infty$. Then, we have the compact embedding $L^{p}_{Q/p}(D)\hookrightarrow L^{q}(D)$ for any smooth bounded domain $D\subset\mathbb{G}$, where $$L^{p}_{Q/p}(D)=\{f\in L^{p}_{Q/p}(\G):\;\;{\rm supp}f\subset D\}.$$
\end{lem}
\begin{proof}[Proof of Lemma \ref{LM: 2.3}] Because of the density argument, it is enough to prove $\mathring{L}^{p}_{Q/p}(D)\hookrightarrow L^{q}(D)$, where $\mathring{L}^{p}_{Q/p}(D)$ denotes the closure of $C_{0}^{\infty}(D)$ with respect to the norm \eqref{norm} with $a=Q/p$.

Let $\phi \in C_{0}^{\infty}(\G)$, with $0\leq \phi \leq 1$, ${\rm supp}\;\phi\subset \overline{B(0,1)}$ and $\int_{\G}\phi(x)dx=1$. We define
\begin{equation}\label{K_epsilon}
K_{\varepsilon}f:=\phi_{\varepsilon}\ast f(x),
\end{equation}
where $f\in L^{1}_{{\rm loc}}(\G)$ and $\phi_{\varepsilon}(x):=\varepsilon^{-Q}\phi(\varepsilon^{-1}x)$ for every $\varepsilon>0$.

We will also use the following lemma, which is an analogue of \cite[Lemma 3.1]{GL-1992} for graded groups.
\begin{lem}\label{LM: auxil-LG} Let $D\subset\mathbb{G}$ be a bounded open set and $1<q<\infty$. Then, $Z\subset L^{q}(D)$ is relatively compact in $L^{q}(D)$, if and only if
\begin{enumerate}
    \item $Z$ is bounded;
    \item $\|K_{\varepsilon}f-f\|_{q}\rightarrow 0$ as $\varepsilon\rightarrow 0$, uniformly in $f\in Z$.
    \end{enumerate}
\end{lem}
\begin{proof}[Proof of Lemma \ref{LM: auxil-LG}] Let us briefly sketch the proof of the lemma. To show the necessity, we extend functions in $L^{q}(D)$ with zero outside $D$. We can take $f_{1},\ldots,f_{n}$ from $Z$ and $r>0$, so that the balls in $L^{q}(D)$ centred at $f_{k}$ with radius $r$ cover $Z$. For a given $f$, let us take $f_{k}$ such that $\|f_{k}-f\|_{q}<r$. Then we have
$$\|f-K_{\varepsilon}f\|_{q}\leq \|f-f_{k}\|_{q}+\|f_{k}-K_{\varepsilon}f_{k}\|_{q}+\|K_{\varepsilon}f_{k}-K_{\varepsilon}f\|_{q}.$$
Taking into account $\|K_{\varepsilon}f\|_{q}\leq \|f\|_{q}$ and $K_{\varepsilon}f\rightarrow f$ in $L^{q}(D)$ as $\varepsilon\rightarrow 0$ also letting $r\rightarrow 0$, we get (2).

We now show sufficiency. Let $f_{n}$ be a bounded sequence in $Z$. By the Banach-Alouglu theorem and the reflexivity of $L^{q}$, $1<q<\infty$, we know that there exists a subsequence, still denoted by $f_{n}$ weakly convergent in $L^{q}$, that is, there exists $f\in L^{q}$ such that
\begin{equation}\label{com_emb_ineq_0}
\int(f_{n}-f)h\rightarrow 0,\;\;\forall h\in L^{q^{\prime}}.
\end{equation}
Let us now show that it actually converges strongly. For this, we write
\begin{equation}\label{com_emb_ineq}
\|f_{n}-f\|_{q}\leq \|f_{n}-K_{\varepsilon}f_{n}\|_{q}+\|K_{\varepsilon}f_{n}-K_{\varepsilon}f\|_{q}+\|K_{\varepsilon}f-f\|_{q}.
\end{equation}
By the assumption (2), we note that the first and third summands vanish when $\varepsilon\rightarrow 0$. For the second summand, \eqref{com_emb_ineq_0} implies that for all $x$ and for all $\varepsilon$ we have $K_{\varepsilon}(f_{n}-f)(x)\rightarrow 0$ as $n\rightarrow \infty$. Since we also have
$$\||K_{\varepsilon}(f_{n}-f)|^{q}\|_{L^{1}}=\|K_{\varepsilon}(f_{n}-f)\|^{q}_{L^{q}}
\leq \|\phi_{\varepsilon}\|^{q}_{L^{1}}\|f_{n}-f\|^{q}_{L^{q}}<\infty,$$
by the Lebesgue dominated convergence theorem we observe that
$$\int|K_{\varepsilon}(f_{n}-f)(x)|^{q}dx\rightarrow 0$$
as $n\rightarrow \infty$.

Thus, from \eqref{com_emb_ineq} we can conclude that $Z$ is relatively compact in $L^{q}(D)$.
\end{proof}
We now come back to the proof of Lemma \ref{LM: 2.3}. Setting $f\equiv 0$ in $\G\backslash D$ for $f\in \mathring{L}^{p}_{Q/p}(D)$, we get a function in $L^{p}_{Q/p}(\G)$. Let us now use Lemma \ref{LM: auxil-LG}. Let $Z$ be a bounded set in $\mathring{L}^{p}_{Q/p}(D)$. Then, using $|B(x,r)|=r^{Q}|B(0,1)|$ for the Haar measure of any open quasi-ball (see e.g. \cite[Page 140]{FR16}) and the critical Gagliardo-Nirenberg inequality \eqref{crit_GN_ineq}, we note that $Z$ is bounded in $L^{q}(D)$.

To complete the proof, it remains to verify the second part of Lemma \ref{LM: auxil-LG}. For $f\in Z$ by denoting
$$
 \psi_{\varepsilon}:=K_{\varepsilon}f-f
$$ and using \eqref{crit_GN_ineq}, we obtain
\begin{equation}\label{RK1}
\|\psi_{\varepsilon}\|_{L^{q}(D)}\leq C \|\psi_{\varepsilon}\|_{\dot{L}^{p}_{Q/p}(D)}^{1-q/p}\|\psi_{\varepsilon}\|^{q/p}_{\dot{L}^{p}(D)}.
\end{equation}
Therefore, it is enough to show that
$$\|K_{\varepsilon}f-f\|_{\dot{L}^{p}_{Q/p}(D)} \rightarrow 0,$$
that is,
$$\|\R^{\frac{Q}{\nu p}}K_{\varepsilon}f-\R^{\frac{Q}{\nu p}}f\|_{L^{p}(D)} \rightarrow 0.$$
Indeed, it holds since by \cite[Part (i) of Lemma 3.1.58]{FR16} we have
$$\R^{\frac{Q}{\nu p}}K_{\varepsilon}f-\R^{\frac{Q}{\nu p}}f= \phi_{\varepsilon} \ast\R^{\frac{Q}{\nu p}}f
-\R^{\frac{Q}{\nu p}}f\rightarrow 0$$
as $\varepsilon\rightarrow 0$.

Therefore, by Lemma \ref{LM: auxil-LG} we can conclude that $Z$ is relatively compact in $L^{q}(D)$.
\end{proof}
We also note the following property of least energy solutions.
\begin{lem}\label{LM: 2.4}
If $v\in\N$ and $\L(v)=d$ then $v$ must be a least energy solution of the Schr\"{o}dinger type equation \eqref{nonlinear}.
\end{lem}
\begin{proof}[Proof of Lemma \ref{LM: 2.4}]  By the Lagrange
multiplier rule there is $\theta\in \mathbb{R}$ such that for any $\psi\in L^{p}_{Q/p}(\mathbb{G})$ we have
\begin{equation}\label{Lag1}
\langle \L'(v), \psi \rangle_{\G}=\theta \langle \I'(v), \psi\rangle_{\G},
\end{equation}
due to the assumption on $v$, where $\langle \cdot, \cdot \rangle_{\G}$ is a dual product between $L^{p}_{Q/p}(\mathbb{G})$ and its dual space.

Taking into account $q>p$, one gets
\begin{equation}\label{Lag2}
\langle \I'(v), v\rangle_{\G} = p \| v \|_{L^{p}_{Q/p}(\mathbb{G})}^{p} - q \int_{\mathbb{G}} |v|^{q} d x=(p-q) \int_{\mathbb{G}}|v|^{q} d x <0.
\end{equation}
On the other hand, we have
\begin{equation}\label{Lag3}
\langle \L'(v), v\rangle_{\G}=\I(v)=0.
\end{equation}
Combining \eqref{Lag2} and \eqref{Lag3}, we obtain $\theta = 0$ from \eqref{Lag1}. It implies that $\L'(v) = 0$. By Definition \ref{def}, we obtain that $v$ is a least energy solution of the nonlinear equation \eqref{nonlinear}.
\end{proof}

Now we are ready to prove Theorem \ref{thm1}.

\begin{proof}[Proof of Theorem \ref{thm1}] We choose $(v_{k})_{k}\subset\N$ as a minimising sequence. By Ekeland variational principle we obtain a sequence $(u_{k})_{k}\subset\N$ satisfying $\L(u_{k})\to d$ and $\L'(u_{k})\to 0$.
Applying the Sobolev inequality and Lemma \ref{LM: 2.2} we see that there exist two positive
constants $A_1$ and $A_2$ with the properties
$$
A_1\leq\|u_{k}\|_{L^{p}_{Q/p}(\mathbb{G})}\leq A_2.
$$
From this and the equality
$$
\|u_{k}\|_{L^{p}_{Q/p}(\mathbb{G})}^{p}=\int_{\G}|u_{k}(x)|^{q}dx,
$$
we obtain the existence of a positive constant $A_3$ so that
\begin{equation}\label{EQ: (2.1)}
\limsup\limits_{k\to\infty}\int_{\G}|u_{k}(x)|^{q}dx\geq A_3 > 0.
\end{equation}
Applying Lemma \ref{LM: 2.3} and the concentration compactness argument of \cite[Lemma 3.1]{ST02}, we have that
 $u_{k}\to0$ in $L^{q}(\G)$ for all $1<p<q<\infty$ if the following
\begin{equation}
\lim\limits_{k\to\infty}\sup\limits_{\eta\in\G}\int\limits_{B(\eta, r)}|u_{k}(x)|^{q} d x = 0
\end{equation}
holds true for some $r > 0$, where $B(\eta, r)$ is a quasi-ball on $\G$ centred at $\eta$ with radius $r$. By \eqref{EQ: (2.1)} there is a constant $A_4>0$ and $r > 1$ such that
\begin{equation}\label{EQ: (2.2)}
\liminf\limits_{k\to\infty}\sup\limits_{\eta\in\G}\int\limits_{B(\eta, r)}|u_{k}(x)|^{q}dx\geq A_4 > 0.
\end{equation}
Assume that there are $\tilde{x}^{k}\in\G$ with
\begin{equation}\label{EQ: (2.3)}
\liminf\limits_{k\to\infty}\int\limits_{B(\tilde{x}^{k}, r)}|u_{k}(x)|^{q}dx\geq \frac{A_4}{2} > 0.
\end{equation}
Taking into account the bi-invariance of the Haar measure and the left invariance of the operator $\R$ one has for all $\tilde{x}^{k}\in\G$
$$
\L(u_{k}(\tilde{x}x))=\L(u_{k}(x))\;\;\text{and}\;\;\I(u_{k}(\tilde{x}x))=\I(u_{k}(x)).$$
Let us denote $\omega_{k}(x):=u_{k}(\tilde{x}x)$. Then we have $\L(\omega_{k})=\L(u_{k})$ and $\I(\omega_{k})=\I(u_{k})$. Moreover, it gives the bounded sequence $(\omega_{k})_{k}$ from $L^{p}_{Q/p}(\mathbb{G})$ with
\begin{equation}\label{EQ: (2.4)}
\liminf\limits_{k\to\infty}\int_{B(0, r)}|\omega_{k}(x)|^{q}dx\geq \frac{A_4}{2} > 0.
\end{equation}
There exists a subsequence, denoted by $\omega_{k}$ that weakly converges to $\phi$ in $L^{p}_{Q/p}(\mathbb{G})$. Then from Lemma \ref{LM: 2.3} we see that $\omega_{k}$ strongly converges to $\phi$ in $L^{q}_{loc}(\G)$. By this and \eqref{EQ: (2.4)}, we have $\phi\neq0$.

Now let us show that $\omega_{k}$ converges strongly to $\phi$ in $L^{p}_{Q/p}(\mathbb{G})$. First we show that $\I(\phi)=0$. We proceed by contradiction. Suppose that $\I(\phi)<0$. Lemma \ref{LM: 2.1} gives that there is a positive number $\mu_{\phi} < 1$ with $\mu_{\phi}\phi\in\N$ for $\I(\phi) < 0$. Since $\I(\omega_{k}) = 0$, applying the Fatou lemma we calculate
$$d+o(1)=\L(\omega_{k})=\left(\frac{1}{p}-\frac{1}{q}\right)\int_{\G} |\omega_{k}(x)|^{q}dx \geq \left(\frac{1}{p}-\frac{1}{q}\right)\int_{\G} |\phi(x)|^{q}dx + o(1), $$
that is,
\begin{equation}\label{EQ: (2.5)}
d+o(1)\geq\left(\frac{1}{p}-\frac{1}{q}\right)\mu_{\phi}^{-q}\int_{\G} |\mu_{\phi}\phi(x)|^{q}dx + o(1)
=\mu_{\phi}^{-q} \L(\mu_{\phi}\phi)  + o(1),
\end{equation}
which implies $d > \L(\mu_{\phi}\phi)$. Since $\mu_{\phi}\phi\in\N$, we arrive at a contradiction.

Suppose now that $\I(\phi)>0$. We need the following lemma:
\begin{lem}[{\cite[Lemma 3]{BL83}}]
\label{BrL_lem}
 Let $\ell:\mathbb{C}\rightarrow\mathbb{R}$ be convex. Then
$$|\ell(a+b)-\ell(a)|\leq \varepsilon[\ell(ma)-m\ell(a)]+|\ell(C_{\varepsilon}b)|+|\ell(-C_{\varepsilon}b)|$$
for any $a,b \in \mathbb{C}$, $0<\varepsilon<\frac{1}{m}<1$ and $\frac{1}{C_{\varepsilon}}=\varepsilon(m-1)$.
\end{lem}
As in \cite[Proof of Theorem 1.2]{CR13}, using this lemma for $\psi_{k}=\omega_{k}-\phi$ we have
$$
0=\I(\omega_{k})=\I(\phi)+\I(\psi_{k})+o(1),
$$
where $\I(\phi) > 0$. Then, one has
\begin{equation}\label{EQ: (2.6)}
\limsup\limits_{k\to\infty} \I(\psi_{k})< 0.
\end{equation}
Applying Lemma \ref{LM: 2.1}, there exists a sequence $\mu_{k}:=\mu_{\psi_{k}}$ with $\mu_{k} \psi_{k}\in \mathcal{N}$ and $\limsup\limits_{k\to\infty}\mu_{k}\in(0, 1)$. Indeed, assume that $\limsup\limits_{k\to\infty}\mu_{k}=1$. Then we have a subsequence $(\mu_{k_{j}})_{j\in\mathbb{N}}$ with the property $\lim\limits_{j\to\infty}\mu_{k_{j}}=1$. Since $\mu_{k_{j}}\psi_{k_{j}}\in \mathcal{N}$ we get that $\I(\psi_{k_{j}})=\I(\mu_{k_{j}}\psi_{k_{j}})+o(1)=o(1)$, which is a contradiction because of \eqref{EQ: (2.6)}. So, we have that $\limsup\limits_{k\to\infty}\mu_{k}\in(0,1)$. A direct calculation gives that
$$
d+o(1)=\L(\omega_{k})=\left(\frac{1}{p}-\frac{1}{q}\right)\int_{\mathbb{G}}|\omega_{k}(x)|^{q}dx
\geq\left(\frac{1}{p}-\frac{1}{q}\right)\int_{\mathbb{G}}|\psi_{k}(x)|^{q}dx,$$
that is,
\begin{equation}\label{EQ: (2.7)}
d+o(1)\geq\left(\frac{1}{p}-\frac{1}{q}\right)\mu_{k}^{-q}\int_{\mathbb{G}}|\mu_{k}\psi_{k}(x)|^{q}dx+o(1)
=\mu_{k}^{-q}\L(\mu_{k}\psi_{k})+o(1).
\end{equation}
Therefore, the fact $\limsup\limits_{k\to\infty}\mu_{k}\in(0,1)$ gives that $d>\L(\mu_{k}\psi_{k})$. It contradicts $\mu_{k}\phi_{k}\in\N$.

Thus, we must have $\I(\phi)=0$. Now we prove that $\psi_{k}=\omega_{k}-\phi\to0$ in the space $L^{p}_{Q/p}(\mathbb{G})$. Indeed, suppose that $\|\psi_{k}\|_{L^{p}_{Q/p}(\mathbb{G})}$ does not vanish as $k\to\infty$. Then we consider the following cases. The first one is when $\int_{\G}|\psi_{k}(x)|^{q}dx$ does not converge to $0$ as $k\to\infty$. Then the following identity
$$
0=\I(\omega_{k})=\I(\phi)+\I(\psi_{k})+o(1)=\I(\psi_{k})+o(1)
$$
with the Brezis-Lieb lemma imply the contradiction
$$d+o(1)=\L(\omega_{k})=\L(\phi)+\L(\psi_{k})+o(1)
\geq d+d+o(1).$$
In the case, when $\int_{\G}|\psi_{k}(x)|^{q}dx$ converges to $0$ as $k\to\infty$, again we obtain a contradiction:
\begin{equation*}
\begin{split}
d+o(1)&=\L(\omega_{k})=\L(\phi)+\frac{1}{p}\|\psi_{k}\|_{L^{p}_{Q/p}(\mathbb{G})}^{p}+o(1)\\
&\geq d+\frac{1}{p}\|\psi_{k}\|_{L^{p}_{Q/p}(\mathbb{G})}^{p}+o(1)>d,
\end{split}
\end{equation*}
when $\|\psi_{k}\|_{L^{p}_{Q/p}(\mathbb{G})}\nrightarrow0$ as $k\to\infty$. Thus, we conclude that $\omega_{k}$ converges strongly to $\phi$ in $L^{p}_{Q/p}(\mathbb{G})$ and $d=\L(\phi)$, where $\phi$ is a least energy solution of \eqref{nonlinear} by Lemma \ref{LM: 2.4}.
\end{proof}

\begin{thm}\label{sharp}
Let $1<p<q<\infty$. Let $\phi$ be a least energy solution of \eqref{nonlinear} and let $C_{GN, \R}$ be the smallest positive constant $C_{1}$ in \eqref{crit_GN_ineq}. Then we have
$$C_{GN, \R}=q^{-q+q/p}\frac{q}{p}\left(\frac{q-p}{p}\right)^{\frac{p-q}{p}}
\|\phi\|_{L^{p}(\mathbb{G})}^{p-q}$$
\begin{equation}\label{sharp1} =
q^{-q+q/p}\frac{q}{p}\left(\frac{q-p}{p}\right)^{\frac{p-q}{p}}
\left(\frac{p^{2}}{q-p}d\right)^{\frac{p-q}{p}},
\end{equation}
where $d$ are defined in \eqref{d}.
\end{thm}
We will use the following lemmas to prove Theorem \ref{sharp}:
\begin{lem}\label{sharp_lem}
Let $\phi$ be a least energy solution of \eqref{nonlinear}. Then we have
\begin{equation}\label{sharp_lem1}\int_{\mathbb{G}}|\mathcal{R}^{\frac{Q}{\nu p}}\phi(x)|^{p}dx=
\frac{q-p}{p}\int_{\mathbb{G}}|\phi(x)|^{p}dx
\end{equation}
and
\begin{equation}\label{sharp_lem2}\int_{\mathbb{G}}|\phi(x)|^{q}dx=\frac{q}{p}
\int_{\mathbb{G}}|\phi(x)|^{p}dx.\end{equation}
Moreover, we have
\begin{equation}\label{sharp_d}
\int_{\mathbb{G}}|\phi(x)|^{p}dx=\frac{p^{2}}{q-p}d.
\end{equation}
\end{lem}
\begin{proof}[Proof of Lemma \ref{sharp_lem}] Since $\phi$ is a least energy solution of \eqref{nonlinear}, we have
\begin{equation}\label{sharp_lem3}
\int_{\mathbb{G}}|\mathcal{R}^{\frac{Q}{\nu p}}\phi(x)|^{p}dx+\int_{\mathbb{G}}|\phi(x)|^{p}dx
=\int_{\mathbb{G}}|\phi(x)|^{q}dx.
\end{equation}
On the other hand, a direct calculation gives for $\lambda>0$ and $\widetilde{\phi}_{\lambda}(x):=\lambda^{\frac{Q}{p}}\phi(\lambda x)$ that
$$\L(\widetilde{\phi}_{\lambda}(x))=\frac{\lambda^{Q}}{p}\int_{\G}|\R^{\frac{Q}{\nu p}}(\phi(\lambda x))|^{p}dx
+\frac{\lambda^{Q}}{p}\int_{\G}|\phi(\lambda x)|^{p}dx
-\frac{\lambda^{\frac{Qq}{p}}}{q}\int_{\G}|\phi(\lambda x)|^{q}dx$$
$$=\frac{\lambda^{Q}}{p}\int_{\G}|\R^{\frac{Q}{\nu p}}\phi(x)|^{p}dx
+\frac{1}{p}\int_{\G}|\phi(x)|^{p}dx
-\frac{\lambda^{\frac{Qq}{p}-Q}}{q}\int_{\G}|\phi(x)|^{q}dx,$$
which implies that
\begin{equation}\label{sharp_lem4} 0=\frac{\partial}{\partial \lambda}\L(\widetilde{\phi}_{\lambda})|_{\lambda=1}=\frac{Q}{p}\int_{\mathbb{G}}|\mathcal{R}^{\frac{Q}{\nu p}}\phi(x)|^{p}dx
-\frac{Q(q-p)}{pq}\int_{\mathbb{G}}|\phi(x)|^{q}dx.
\end{equation}
Thus, from \eqref{sharp_lem3} and \eqref{sharp_lem4} we obtain \eqref{sharp_lem1} and \eqref{sharp_lem2}, respectively.

In order to show \eqref{sharp_d}, using \eqref{L} and \eqref{sharp_lem2}, we calculate
\begin{multline*}
d=\L(\phi)=\frac{1}{p}\|\phi\|_{L^{p}_{Q/p}(\mathbb{G})}^{p}-\frac{1}{q}
\|\phi\|^{q}_{L^{q}(\mathbb{G})}
=\left(\frac{1}{p}-\frac{1}{q}\right)\|\phi\|^{q}_
{L^{q}(\mathbb{G})}=\frac{q-p}{p^{2}}
\int_{\mathbb{G}}|\phi(x)|^{p}dx.
\end{multline*}
Then, it follows that
\begin{equation*}%\label{sharp_d}
\int_{\mathbb{G}}|\phi(x)|^{p}dx=\frac{p^{2}}{q-p}d,
\end{equation*}
which is \eqref{sharp_d}.
\end{proof}

\begin{lem}\label{sharp_lemma} Let $T_{\rho, p, q}$ is defined as
\begin{equation}\label{T_rho}
T_{\rho, p, q}:=\inf\left\{\|u\|_{L^{p}_{Q/p}(\mathbb{G})}^{p}:u\in L^{p}_{Q/p}(\mathbb{G}) \;\;{\rm and}\;\; \int_{\mathbb{G}}|u(x)|^{q}dx=\rho\right\},
\end{equation}
for $\rho>0$. If $\phi$ is a minimiser obtained in Theorem \ref{thm1}, then $\phi$ is a minimiser of $T_{\rho_{0}, p, q}$ such that $\rho_{0}=\int_{\mathbb{G}}|\phi(x)|^{q}dx$.
\end{lem}
\begin{proof}[Proof of Lemma \ref{sharp_lemma}]
From the definition of $T_{\rho_{0}, p, q}$, one has $\|\phi\|_{L^{p}_{Q/p}(\mathbb{G})}^{p}\geq T_{\rho_{0}, p, q}$. We claim that $T_{\rho_{0}, p, q}\geq\|\phi\|_{L^{p}_{Q/p}(\mathbb{G})}^{p}$. Indeed, using Lemma \ref{LM: 2.1} for any $u\in L^{p}_{Q/p}(\mathbb{G})$ satisfying $\int_{\mathbb{G}}|u(x)|^{q}dx=\int_{\mathbb{G}}|\phi(x)|^{q}dx$ there is a unique
$$\lambda_{0}=\|u\|_{L^{p}_{Q/p}(\mathbb{G})}^{\frac{p}{q-p}}\left(\int_{\G}|u|^{q}dx\right)^{-\frac{1}{q-p}}$$
with $\I(\lambda_{0}u)=0$. Since we know that $\lambda_{0}u\neq0$ and $\phi$ achieves the minimum $d$, then a direct calculation gives
$$\left(\frac{1}{p}-\frac{1}{q}\right)\|\phi\|_{L^{p}_{Q/p}(\mathbb{G})}^{p}=\L(\phi)\leq \L(\lambda_{0}u)=\left(\frac{1}{p}-\frac{1}{q}\right)\lambda_{0}^{p}\|u\|_{L^{p}_{Q/p}(\mathbb{G})}^{p}$$
$$=\left(\frac{1}{p}-\frac{1}{q}\right)\|u\|_{L^{p}_{Q/p}(\mathbb{G})}
^{\frac{p^{2}}{q-p}}\left(\int_{\mathbb{G}}|u(x)|^{q}dx\right)^{-\frac{p}{q-p}}\|u\|_{L^{p}_{Q/p}(\mathbb{G})}^{p}.$$
Then, from $\int_{\mathbb{G}}|u(x)|^{q}dx=\int_{\mathbb{G}}|\phi(x)|^{q}dx$ and $\int_{\mathbb{G}}|\phi(x)|^{q}dx=\|\phi\|_{L^{p}_{Q/p}(\mathbb{G})}^{p}$, we obtain $\|\phi\|_{L^{p}_{Q/p}(\mathbb{G})}^{p}\leq\|u\|_{L^{p}_{Q/p}(\mathbb{G})}^{p}$. From the arbitrariness of the function $u$ one gets $T_{\rho_{0}, p, q}\geq\|\phi\|_{L^{p}_{Q/p}(\mathbb{G})}^{p}$. Thus, by $T_{\rho_{0}, p, q}=\|\phi\|_{L^{p}_{Q/p}(\mathbb{G})}^{p}$, we conclude that $\phi$ is a minimiser of $T_{\rho_{0}, p, q}$.
\end{proof}
Now let us prove Theorem \ref{sharp}.
\begin{proof}[Proof of Theorem \ref{sharp}] For $u\neq0$ we define
\begin{equation}\label{J}
J(u):=q^{q-q/p}\left(\int_{\mathbb{G}}|\mathcal{R}^{\frac{Q}{\nu p}}u(x)|^{p}dx\right)^{\frac{q-p}{p}}
\left(\int_{\mathbb{G}}|u(x)|^{p}dx\right)\left(\int_{\mathbb{G}}|u(x)|^{q}dx\right)^{-1}.
\end{equation}
We estimate the sharp expression $C_{GN, \R}$ by studying the following minimisation problem
$$C_{GN, \R}^{-1}=\inf\{J(u):u\in L^{p}_{Q/p}(\mathbb{G})\ \backslash\{0\}\}.$$
Taking into account $\phi(x)\neq0$ and $\phi \in  L^{p}_{Q/p}(\mathbb{G})$, and using  Lemma \ref{sharp_lem} we get
$$J(\phi)=q^{q-q/p-1}p\left(\frac{q-p}{p}\right)^{\frac{q-p}{p}}
\left(\int_{\mathbb{G}}|\phi(x)|^{p}dx\right)^{\frac{q-p}{p}},$$
which gives that
\begin{equation}\label{sharp_thm001}
C_{GN, \R}^{-1}\leq q^{q-q/p-1}p\left(\frac{q-p}{p}\right)^{\frac{q-p}{p}}
\left(\int_{\mathbb{G}}|\phi(x)|^{p}dx\right)^{\frac{q-p}{p}}.
\end{equation}
Now we obtain a lower estimate for $C_{GN, \R}^{-1}$. We denote $\omega(x):=\lambda u(\mu x)$ for $\lambda, \mu>0$ and for all $u\in L^{p}_{Q/p}(\mathbb{G})\ \backslash\{0\}$. Then, we calculate
\begin{equation}\label{sharp_thm01}
\int_{\mathbb{G}}|\mathcal{R}^{\frac{Q}{\nu p}}\omega(x)|^{p}dx=\lambda^{p}\int_{\mathbb{G}}
|\mathcal{R}^{\frac{Q}{\nu p}}u(x)|^{p}dx,
\end{equation}
$$\int_{\mathbb{G}}|\omega(x)|^{p}dx=\lambda^{p}\mu^{-Q}\int_{\mathbb{G}}
|u(x)|^{p}dx$$
and
$$\int_{\mathbb{G}}|\omega(x)|^{q}dx=\lambda^{q}\mu^{-Q}\int_{\mathbb{G}}|u(x)|^{q}dx.$$
Choosing $\lambda$ and $\mu$ such that
\begin{equation}\label{sharp_thm1}\lambda^{p}\mu^{-Q}\int_{\mathbb{G}}
|u(x)|^{p}dx=\int_{\mathbb{G}}|\phi(x)|^{p}dx
\end{equation}
and
\begin{equation}\label{sharp_thm2}\lambda^{q}\mu^{-Q}\int_{\mathbb{G}}|u(x)|^{q}dx=\int_{\mathbb{G}}|\phi(x)|^{q}dx,\end{equation}
and using \eqref{sharp_lem2}, we obtain
$$\lambda^{q}\mu^{-Q}\int_{\mathbb{G}}|u(x)|^{q}dx=\int_{\mathbb{G}}|\phi(x)|^{q}dx=\frac{q}
{p}\int_{\mathbb{G}}|\phi(x)|^{p}dx$$
$$=\frac{q}
{p}\lambda^{p}\mu^{-Q}\int_{\mathbb{G}}|u(x)|^{p}dx,$$
which implies that
\begin{equation}\label{sharp_thm3}
\lambda^{p}=\left(\frac{q}
{p}\right)^{\frac{p}{q-p}}
\left(\int_{\mathbb{G}}|u(x)|^{p}dx\right)^{\frac{p}{q-p}}
\left(\int_{\mathbb{G}}|u(x)|^{q}dx\right)^{-\frac{p}{q-p}}.
\end{equation}
From \eqref{sharp_thm1} and \eqref{sharp_thm2} one has
$$
\int_{\mathbb{G}}|\omega(x)|^{p}dx=\int_{\mathbb{G}}|\phi(x)|^{p}dx \;\textrm{ and }\;\int_{\mathbb{G}}|\omega(x)|^{q}dx=\int_{\mathbb{G}}|\phi(x)|^{q}dx.
$$
Since $\phi$ is a minimiser of $T_{\rho_{0}, p, q}$ such that $\rho_{0}=\int_{\mathbb{G}}|\phi(x)|^{q}dx$, then using Lemma \ref{sharp_lemma}, we have
$$\int_{\mathbb{G}}|\mathcal{R}^{\frac{Q}{\nu p}}\omega(x)|^{p}dx\geq
\int_{\mathbb{G}}|\mathcal{R}^{\frac{Q}{\nu p}}\phi(x)|^{p}dx.$$
Then by \eqref{sharp_thm01} and \eqref{sharp_lem1}, one calculates
$$\lambda^{p}\int_{\mathbb{G}}
|\mathcal{R}^{\frac{Q}{\nu p}}u(x)|^{p}dx=\int_{\mathbb{G}}|\mathcal{R}^{\frac{Q}{\nu p}}\omega(x)|^{p}dx$$
$$\geq\int_{\mathbb{G}}|\mathcal{R}^{\frac{Q}{\nu p}}\phi(x)|^{p}dx=
\frac{q-p}{p}\int_{\mathbb{G}}|\phi(x)|^{p}dx.$$
Plugging here \eqref{sharp_thm3}, we obtain that
$$\left(\frac{q}
{p}\right)^{\frac{p}{q-p}}
\left(\int_{\mathbb{G}}|u(x)|^{p}dx\right)^{\frac{p}{q-p}}$$
$$\times\left(\int_{\mathbb{G}}|u(x)|^{q}dx\right)^{-\frac{p}{q-p}}
\int_{\G}|\mathcal{R}^{\frac{Q}{\nu p}}u(x)|^{p}dx\geq \frac{q-p}{p}
\int_{\mathbb{G}}|\phi(x)|^{p}dx.$$
Now, by the definition of $J(u)$ in \eqref{J}, we arrive at
$$J(u)\geq q^{q-q/p-1}p\left(\frac{q-p}{p}\right)^{\frac{q-p}{p}}
\left(\int_{\mathbb{G}}|\phi(x)|^{p}dx\right)^{\frac{q-p}{p}}.$$
Using again the arbitrariness of the function $u$, we get
\begin{equation}\label{sharp_thm5}C_{GN, \R}^{-1}\geq q^{q-q/p-1}p\left(\frac{q-p}{p}\right)^{\frac{q-p}{p}}
\left(\int_{\mathbb{G}}|\phi(x)|^{p}dx\right)^{\frac{q-p}{p}}.
\end{equation}
Combining \eqref{sharp_thm001} and \eqref{sharp_thm5}, one obtains the first equality in \eqref{sharp1}. Then, this first equality with \eqref{sharp_d} imply the second equality in \eqref{sharp1}.
\end{proof}

\end{document}